%% file: ugb_preprint_02.tex
\documentclass[a4paper,10pt]{amsart}

\usepackage{amsthm,amssymb,amsmath}

\usepackage[british]{babel}
\usepackage{enumerate}
\usepackage{subfigure}
\usepackage[usenames,dvipsnames]{color}
\usepackage[dvips]{graphicx}
\usepackage[dvipsnames]{pstricks}
\usepackage{pstricks-add}
\usepackage{float}
\usepackage{microtype}
\usepackage{frcursive}
\usepackage{mathcomp} 
\usepackage{textcomp} 

\numberwithin{equation}{section}

\theoremstyle{definition}
\newtheorem{dfn}{Definition}[section]

\newtheorem{cnv}[dfn]{Convention}

\newtheorem{ntn}[dfn]{Notation}
\newtheorem{rec}[dfn]{Reminder}
\newtheorem{rec+dfn}[dfn]{Reminder \& Definition}
\newtheorem{dfn+rem}[dfn]{Definition \& Remark}
\newtheorem{rem+dfn}[dfn]{Remark \& Definition}
\newtheorem{rem}[dfn]{Remark}
\newtheorem{ntn+rem}[dfn]{Notation \& Remark}
\newtheorem{exa}[dfn]{Example}

\theoremstyle{plain}
\newtheorem{thm}[dfn]{Theorem}
\newtheorem{lem}[dfn]{Lemma}
\newtheorem{cor}[dfn]{Corollary}
\newtheorem{pro}[dfn]{Proposition}

\newcommand{\fraktur}[1]{\mathfrak{#1}}
\DeclareMathOperator{\LLOO}{\smash{\text{\it\textcursive{min}}}}
\DeclareMathOperator{\TTOO}{\smash{\text{\it\textcursive{lm}}}}

\DeclareMathOperator{\SO}{FO}
\DeclareMathOperator{\CO}{CO}
\DeclareMathOperator{\TO}{TO}
\DeclareMathOperator{\WO}{WO}
\DeclareMathOperator{\AO}{AO}
\DeclareMathOperator{\DO}{DO}
\DeclareMathOperator{\DCO}{DCO}
\DeclareMathOperator{\MIN}{min}

\DeclareMathOperator{\LC}{LC}
\DeclareMathOperator{\lc}{lc}

\DeclareMathOperator{\lt}{lm}
\DeclareMathOperator{\lm}{lm}
\DeclareMathOperator{\LT}{LM}
\DeclareMathOperator{\LM}{LM}
\DeclareMathOperator{\HF}{HF}
\DeclareMathOperator{\HP}{HP}
\DeclareMathOperator{\ind}{\varrho}
\DeclareMathOperator{\Img}{Im}

\DeclareMathOperator{\Supp}{Supp}
\DeclareMathOperator{\supp}{supp}

\DeclareMathOperator{\len}{len}

\renewcommand{\mod}{\mathop{/}}
\renewcommand{\epsilon}{\varepsilon}
\renewcommand{\varsigma}{\sigma}
\renewcommand{\imath}{\iota}

\renewcommand{\Im}{\Img}

\newcommand{\NN}{\mathbb{N}}

\newcommand{\RR}{\mathbb{R}}

\newcommand{\wo}{\smallsetminus}

\renewcommand{\emptyset}{\varnothing}
\newcommand{\notion}[1]{\textit{#1}}

\newcommand{\res}{ \text{$\upharpoonright$} }
\newcommand{\und}{\,\mathop{\wedge}\,}
\newcommand{\oder}{\,\mathop{\vee}\,}
\newcommand{\dann}{\Rightarrow}

\newcommand{\gdw}{\Leftrightarrow}

\newcommand{\slGamma}{\mathit{\Gamma}}

\newcommand{\slLambda}{\mathit{\Lambda}}

\newcommand{\Nu}{N}

\newcommand{\slPhi}{\mathit{\Phi}}
\newcommand{\slPsi}{\mathit{\Psi}}

\newcommand{\formatenum}{} 

\begin{document}

\frenchspacing
\baselineskip = 1.375\baselineskip
\relpenalty = 9999 
\binoppenalty = 10000
\hyphenpenalty = 0

\title{A topological approch to leading monomial ideals}
\author{Roberto Boldini} 
\address{%
Universit\"at Z\"urich,
Institut f\"ur Mathematik,
Winterthurerstr.~190, 
CH-8057 Z\"urich} 
\email{roberto.boldini@math.uzh.ch}
\thanks{I am very grateful to Prof.~Dr.~Markus Brodmann of the University of Zurich,
Switzerland,
for his precious comments and his gracious encouragement during the writing of this paper.
My acknowledgements also to Prof.~Dr.~Matthias Aschenbrenner of the University of California,
Los Angeles, for his kind communication of the proof of Theorem \ref{compact}.}
\subjclass[2010]{Primary 13C05 13C13 16D25}
\date{19th May 2010}






\begin{abstract}
We define a very natural topology on the set of total orderings of monomials of any algebra
having a countable basis over a field.
This topological space and some notable subspaces are compact.

This topological framework
allows us to deduce some finiteness results about leading monomial
ideals of any fixed ideal, namely:
(1)~the number of minimal leading monomial ideals with respect to total orderings is finite;
(2)~the number of leading monomial ideals with respect to degree orderings is finite;
(3)~the number of leading monomial ideals with respect to admissible orderings is finite
under some multiplicativity assumptions on the considered algebra.

Finally we  are able to infer the existence of universal Gr\"obner bases
from the topological properties of degree and admissible orderings
in a class of algebras that includes at least the algebras of solvable type.
These existence results turn out to be independent from the finiteness results mentioned above,
in contrast to the typical situation that occurs with ``classical'' more combinatorial proofs.
\end{abstract}









\maketitle



\input{phd_1_commalg.tex}



\end{document}

%% file: phd_1_commalg.tex
\section*{Introduction}

\noindent
In this paper we deal with leading monomial ideals of ideals in
some classes of algebras over a field
with respect to several sorts of total orderings on their bases, whose elements we call
monomials.

We introduce a topology on the set of all total orderings of monomials.
It turns out that
the so obtained topological space  is compact
and, in the case of countable bases,
this topology is precisely the one induced by a very natural metric on
such total orderings.
In virtue of this fact,
after showing that certain kinds of total orderings
build closed subsets and hence are compact subspaces,
and by considering certain quotient spaces 
(with respect to an appropriate equivalence relation)
which turn out to be discrete,
we are able to prove some finiteness results about leading monomial ideals of such algebras, namely:
if $A$ is an algebra over a field $K$ such that $A$ has a countable basis as a free $K$-module,
and if $H$ is any subset of $A$,
then:
\begin{enumerate}[\upshape (1)]\formatenum
\item\label{11}
the number of minimal leading monomial ideals of $H$ with respect to total orderings of
monomials of $A$ is finite,
see Theorem \ref{fini2},
\item\label{22}
the number of leading monomial ideals of $H$ with respect to degree orderings of monomials
of $A$ is finite,
see Theorem \ref{fini3b},
\item\label{33}
the number of leading monomial ideals of $H$ with respect to admissible orderings of monomials
of $A$ is finite whenever $H$  is a (left, right, or two-sided) ideal of $A$ and
$A$ satisfies two multiplicativity conditions, namely, $A$ is a domain and $A$ behaves
multiplicatively on taking leading monomials with respect to admissible orderings,
see Theorem \ref{cocco}.
\end{enumerate}

Carrying on with this topological approach, generalizing
\cite{Sch} and \cite{Sik},
we prove that every (left, right, or two-sided) ideal $J$ of $A$ admits a
$\fraktur{T}\text{-}$universal Gr\"obner basis $U$, that is, $U$ is a Gr\"obner basis of $J$
with respect to each ${\preceq}\in\fraktur{T}$, where $\fraktur{T}$ is a closed subset
of the set of all degree orderings or of all admissible orderings of monomials of $A$.

Statements about the existence of universal Gr\"obner bases, 
for instance in the context of commutative polynomial rings over a field,
are usually infered from a finiteness result similar to \eqref{33}
and from the availability of a division algorithm
by which one can construct reduced Gr\"obner bases, 
a selected finite union of which is then a universal Gr\"obner basis, see \cite{Stu}.

We shall see that, actually,
the topological properties of the considered
spaces of total orderings of monomials, above all
compactness, are sufficient to prove the existence of universal Gr\"obner bases,
even in the more general context treated here.

The algebras on which these results can be applied comprehend at least the algebras of solvable 
type and the enveloping algebras of finite-dimensional Lie algebras.
Some of our results, such as \eqref{33} and the existence of universal Gr\"obner bases in the
just mentioned classes of algebras, are not new, see \cite{WeiUGB} for instance.
New are, in our knowledge, \eqref{11} and~\eqref{22}.

Through \eqref{11} one gains a new insight why
there exist only finitely many leading monomial
ideals of a given ideal with respect to admissible orderings (Theorem \ref{cocco}).
Indeed, 
there exist at most finitely many minimal
such ideals at all with respect to any closed subset of total orderings
(Theorem \ref{fini2}), and the admissible orderings form a closed subset
(Proposition~\ref{AO(N) compact})
and force leading monomial ideals to be minimal
(Corollary~\ref{lollo} of the Macaulay Basis Theorem~\ref{macaulay}).

Through \eqref{22} one gets a deeper intuition why one finds only finitely many 
leading monomial ideals of a given ideal with respect to degree-compatible
orderings (Remark \ref{dco endlich}).
Indeed, degree preservation on taking leading monomials alone
without the com\-pa\-ti\-bi\-li\-ty axiom already implies this behaviour (Theorem~\ref{fini3b}).

Our intention has been also to push the topological methods introduced in
\cite{Sch} and \cite{Sik} to the case
of some further orderings than only admissible ones
and
of some non\-commutative algebras.
Beside the mentioned finiteness results, we have obtained a sort of
\emph{topological framework for orderings of monomials},
which we were able to successfully apply to the study of 
leading monomial ideals
and
universal Gr\"obner bases.
Furthermore, some relations among different kinds of orderings was put to evidence.
Beside those already mentioned, two further topological phenomena came to light:
\begin{enumerate}[\upshape (1)]\formatenum
\addtocounter{enumi}{3}
\item
there exist ``few'' degree-compatible orderings,
that is,
precisely,
the degree-compatible orderings are nowhere dense among the degree orderings,
clearly except for the case of univariate polynomials,
see Proposition \ref{nowhere dense} and Remark~\ref{T1},
\item
there is a  relation between topological density and the possibility
to find a universal Gr\"obner basis,
see Remark \ref{denso1}, Lemma \ref{denso2} and Example~\ref{denso3}.
\end{enumerate}

We conclude by saying that
remarkable benefits of the topological approach are, in our opinion,
the high level of generality and the simplicity of the argumentations.
A drawback, at least at first sight, is the non\-constructivity of the proofs.
But who knows? See~\ref{denso1}.

\section*{R\'esum\'e}

\noindent
In \cite{Sik}, for semigroups $S$,
Sikora introduced a  natural topology $\mathcal{U}(S)$
on the set $\TO(S)$ of the total orderings on $S$ and proved that $\TO(S)$ is compact
with respect to~$\,\mathcal{U}(S)$.
This can be done actually for any set $S$.

We start with a polynomial ring $K[X]=K[X_1,\ldots,X_t]$ over a field $K$, where $t\in\NN$,
and with several sorts of total orderings on the set $M=\{X^\nu\mid\nu\in\NN_0^t\}$ of the
monomials of $K[X]$, namely, we consider the following subsets of $\TO(M)$:
\begin{enumerate}[\upshape (1)]\formatenum
\item
the set $\WO(M)$ of the total well-orderings on $M$;
\item
the set $\SO_1(M)=\{{\leq}\in\TO(M)\mid m\in M\Rightarrow 1\leq m\}$ of the $1$-founded
orderings on~$M$;
\item
the set $\CO(M)=\{{\leq}\in\TO(M)\mid X^\upsilon\leq X^\nu\Rightarrow
X^{\upsilon+\gamma}\le X^{\nu+\gamma}\}$ of the compatible orderings, or semigroup orderings, 
on~$M$;
\item
the set $\DO(M)=\{{\leq}\in\TO(M)\mid p\in K[X]\Rightarrow \deg(p)=\deg(\LT_\leq(p))\}$ of 
the degree orderings on~$M$;
\item
the set $\AO(M)=\SO_1(M)\cap\CO(M)$ of the admissible orderings, or monoid orderings, on $M$;
\item
the set $\DCO(M)=\DO(M)\cap\CO(M)$ of the degree-compatible orderings on~$M$.
\end{enumerate}
Then we have the following results:
\begin{enumerate}[\upshape (1)]\formatenum
\item $\SO_1(M)$ is closed in $\TO(M)$;
\item $\CO(M)$ is closed in $\TO(M)$;
\item $\DO(M)$ is closed in $\TO(M)$ and $\DO(M)\subseteq\WO(M)\cap\SO_1(M)$;
\item $\AO(M)$ is closed in $\TO(M)$ and $\AO(M)=\WO(M)\cap\CO(M)$;
\item $\DCO(M)$ is closed in $\TO(M)$;
\item $\DCO(M)$ is nowhere dense in $\DO(M)$ if $t>1$, otherwise $\DCO(M)=\DO(M)$.
\end{enumerate}
The Venn diagram in Figure~\ref{venn} sketches the situation.
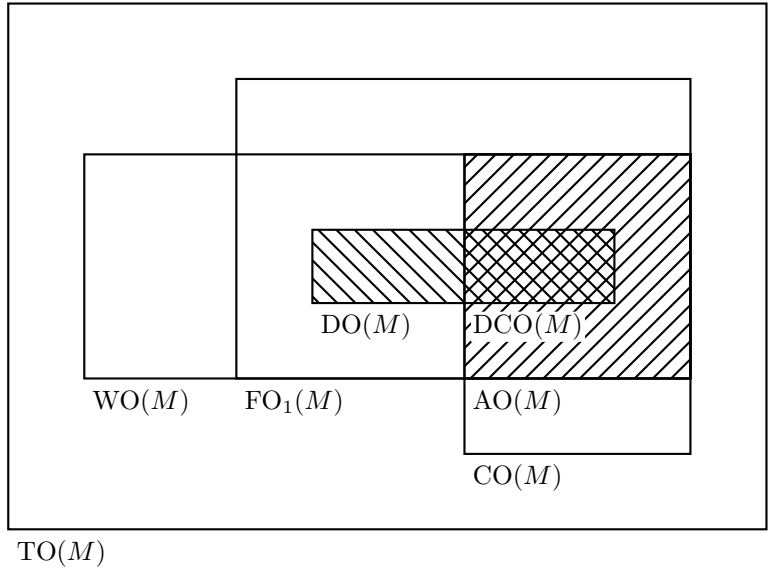
\begin{figure}[ht]
\centering
\input{fig_subspaces.tex}
\caption{Subspaces of total orderings of monomials}%
\label{venn}
\end{figure}

After these preliminaries, given any $\fraktur{S}\subseteq\TO(M)$ and any $E\subseteq K[X]$,
we consider the set
$\TTOO_\fraktur{S}(E)=\{\LM_\leq(E)\mid {\leq}\in\fraktur{S}\}$
of the leading monomial ideals $\LM_\leq(E)$ of $E$ with respect to the
total orderings ${\leq}\in\fraktur{S}$ and the set
$\LLOO_\fraktur{S}(E)$ of the minimal elements of $\TTOO_\fraktur{S}(E)$ with respect
to the inclusion relation~${\subseteq}$, and show that $\LLOO_\fraktur{S}(E)$ is  finite  if $\fraktur{S}$
is closed in $\TO(M)$.

The proof goes as follows.
The set $\min_E(\fraktur{S})$ of the elements ${\leq}\in\fraktur{S}$
such that $\LM_\leq(E)$ is ${\subseteq}$-minimal in $\TTOO_\fraktur{S}(E)$
is closed in $\fraktur{S}$,
and hence $\min_E(\fraktur{S})$ is compact under our hypo\-thesis on~$\fraktur{S}$.
Thus the quotient space $\min_E(\fraktur{S}) \mod {\sim}_E$ of $\min_E(\fraktur{S})$, where
${\leq} \sim_E {\leq'}$ if and only if $\LM_{\leq}(E)=\LM_{\leq'}(E)$, is compact.
Since $\min_E(\fraktur{S}) \mod {\sim}_E$ is also discrete, it follows that
$\min_E(\fraktur{S}) \mod {\sim}_E$ is finite.
Of course, there exists a canonical bijection between $\min_E(\fraktur{S}) \mod {\sim}_E$
and $\LLOO_\fraktur{S}(E)$.

Now we turn our attention to degree orderings. 
$\DO(M)$ and $\DO(M)\mod{\sim}_E$ are compact.
We show by means of Hilbert functions that $\DO(M)\mod{\sim}_E$ is discrete and hence finite.
Thus $\TTOO_{\DO(M)}(E)$ is finite, that is, there exist at most finitely many
leading monomial ideals of $E$ from degree orderings.
The idea of applying Hilbert functions in such ``topological contexts'' was already used 
in a similar manner by Schwartz in~\cite{Sch} in the case of admissible orderings.

When considering closed subsets $\fraktur{S}$ of $\AO(M)$, we obtain a similar and well-known
finiteness result. Indeed, in this case, if $I$ is an ideal of $K[X]$,
the Macaulay Basis Theorem holds and comes to our aid as it implies that
$\TTOO_\fraktur{S}(I)=\LLOO_\fraktur{S}(I)$, which we already know to be finite.

Next let $\slPhi$ be a $K$-module isomorphism of $V$ in $K[X]$ and consider
the $K$-basis $N=\slPhi^{-1}(M)$ of $V$.
Then $\slPhi$ induces a homeomorphism $\phi$ of $\TO(N)$ in $\TO(M)$.
Now, given a total ordering ${\preceq}$ on $N$, we may speak of the ${\preceq}$-leading  
component $\lm_{\preceq}(v)\in N$ in the unique representation $v=\sum_{n\in N}c_n n$
with $c_n\in K\wo\{0\}$ of any element $v\in V$ as a $K$-linear combination over $N$.
Further, given $H\subseteq V$, we consider the ideal
\begin{equation*}
\LM_{\preceq}(H)=
\langle \slPhi(\lm_{\preceq}(h))\mid h\in H\rangle
=\langle \LM_{\phi({\preceq})}(\slPhi(h))\mid h\in H\rangle
\end{equation*}
of $K[X]$.
For all  $H\subseteq V$,  $E\subseteq K[X]$,  ${\preceq}\in\TO(N)$,
 ${\leq}\in\TO(M)$,  $\fraktur{T}\subseteq\TO(N)$,  $\fraktur{S}\subseteq\TO(M)$
 we have:
\begin{enumerate}[\upshape (1)]\formatenum
\item
$\LM_{\preceq}(H)=\LM_{\phi({\preceq})}(\slPhi(H))$
and
$\LM_\leq(E)=\LM_{\phi^{-1}({\leq})}(\slPhi^{-1}(E))$;
\item
$\TTOO_\fraktur{T}(H)=\TTOO_{\phi(\fraktur{T})}(\slPhi(H))$
and
$\TTOO_{\fraktur{S}}(E)=\TTOO_{\phi^{-1}(\fraktur{S})}(\slPhi^{-1}(E))$;
\item
$\LLOO_\fraktur{T}(H)=\LLOO_{\phi(\fraktur{T})}(\slPhi(H))$
and
$\LLOO_{\fraktur{S}}(E)=\LLOO_{\phi^{-1}(\fraktur{S})}(\slPhi^{-1}(E))$.
\end{enumerate}
Thus what we have said above about $K[X]$ easily translates to $V$.
With one exception:
assuming that $\fraktur{T}$ is closed in $\AO(N)$,
the equality $\TTOO_\fraktur{T}(H)=\LLOO_\fraktur{T}(H)$
holds so far only under the hypothesis that $H=\slPhi^{-1}(I)$ for some ideal $I$
of~$K[X]$.

Therefore, when considering the set $\AO(N)=\phi^{-1}(\AO(M))$ of the admissible orderings on
$N$, we replace the $K$-module $V$ by an associative but not necessarily commutative
$K$-algebra $A$ that is a domain and is isomorphic to $K[X]$ as a $K$-module.
Assuming similar multiplicativity properties of $A$ on taking leading monomials
as in the case of $K[X]$,
we prove a generalized version of the Macaulay Basis Theorem, which then implies
the equality $\TTOO_\fraktur{T}(J)=\LLOO_\fraktur{T}(J)$
for each closed $\fraktur{T}\subseteq\AO(N)$
and each (left, right, two-sided) ideal $J\subseteq A$.

Finally, for a $K$-algebra $A$ isomorphic to $K[X]$ as a  $K$-module,
following this topological approach and applying the results obtained so far,
we show that every (left, right, two-sided) ideal of $A$ admits a
$\fraktur{T}$-universal Gr\"obner basis,
where $\fraktur{T}$ is any closed subset of $\DO(N)$.
To prove a similar result for closed subsets $\fraktur{T}$ of $\AO(N)$,
we have to require that $A$ is a domain and is multiplicative on taking leading monomials
over~$\fraktur{T}$.

As mentioned before, our proofs of theorems about universal Gr\"obner bases
do not rely on the finiteness of the total number of
leading monomial ideals.
Indeed, the statements about universal Gr\"obner bases as well as the finiteness results
both descend directly from some of the topological properties of total orderings and,
partly, from the generalized Macaulay Basis Theorem.

\section*{General remark}

\noindent
In this paper all the statements involving ideals of non\-commutative rings are proved only
for left ideals.
These statements translate word by word to right and two-sided ideals, too.

\section{Topological spaces of total orderings on sets}\label{chp1sec1}

\noindent
In this section, let $S$ be a set.

\begin{dfn}\label{def 1}
A \notion{total ordering} on $S$ is a binary relation ${\preceq}$ on $S$ such that
it holds
antisymmetry: $a \preceq b \und b \preceq a \dann a = b$, 
transitivity: $a \preceq b \und b \preceq c \dann a \preceq c$, 
totality: $a \preceq b \oder b \preceq a$, 
for all $a,b,c\in S$.
Totality implies reflexivity: $a \preceq a$ for all $a\in S$.
The non\-empty set of all total orderings on $S$ is denoted  $\TO(S)$.

Given any ordered pair $(a,b)\in S\times S$,
let $\fraktur{U}_{(a,b)}$ be the set of all total orderings ${\preceq}$ on $S$
for which $a \preceq b$.
Let $\mathcal{U}(S)$ be the coarsest topology of $S$ for which all the sets
$\fraktur{U}_{(a,b)}$ are open.
This is the topology
for which $\{\fraktur{U}_{(a,b)}\mid (a,b)\in S\times S\}$ is a subbasis, 
that is, the open sets in $\mathcal{U}(S)$ are precisely the unions of finite intersections 
of sets of the form $\fraktur{U}_{(a,b)}$.
Observe that $\fraktur{U}_{(a,a)}=\TO(S)$ and that 
$\fraktur{U}_{(a,b)}=\TO(S)\wo \fraktur{U}_{(b,a)}$ if~$a\ne b$,
so that the  sets $\fraktur{U}_{(a,b)}$ 
are also closed.

Let $\mathbf{S}$ be any \notion{filtration} of $S$,
that is, $\mathbf{S}=(S_i)_{i\in\NN_0}$ is a family of subsets $S_i$ of $S$ such that
(a)~${S_{0}=\emptyset}$,
(b)~${S_{i}\subseteq S_{i+1}}$ for all ${i\in\NN_0}$, 
(c)~$S=\bigcup_{i\in\NN_0}S_i$.
Let us define the function
${d_\mathbf{S}:\TO(S)\times\TO(S)\rightarrow\RR}$ by the rule 
${d_\mathbf{S}({\preceq'},{\preceq''})= 2^{-r}}$ where 
$\smash{r=\sup{}\{i\in\NN_0\mid {\preceq'}\res_{S_i} = {\preceq''}\res_{S_i}\}}$.
He\-re ${\res}$ denotes restriction.
First of all, we have $\{0\} \subseteq \Im(d_\mathbf{S}) \subseteq [0,1]$.
Because $\mathbf{S}$ is exhaustive by~(c),
it holds ${d_\mathbf{S}({\preceq'},{\preceq''})=0}$ if and only if
${\preceq'} = {\preceq''}$.
Further, ${d_\mathbf{S}({\preceq'},{\preceq''})}={d_\mathbf{S}({\preceq''},{\preceq'})}$.
Finally, 
${d_\mathbf{S}({\preceq'},{\preceq'''})}
\le {d_\mathbf{S}({\preceq'},{\preceq''})}+{d_\mathbf{S}({\preceq''},{\preceq'''})}$,
since 
${d_\mathbf{S}({\preceq'},{\preceq'''})}
\le \max{}\{{d_\mathbf{S}({\preceq'},{\preceq''})},{d_\mathbf{S}({\preceq''},{\preceq'''})}\}$.
Thus $d_\mathbf{S}$ is a metric on $\TO(S)$, dependent on the choice of the filtration 
$\mathbf{S}$ of $S$.
\end{dfn}

\begin{thm}\label{U=N}
Assume that there exists a filtration $\mathbf{S}=(S_i)_{i\in\NN_0}$ of $S$
such that each of the sets $S_i$ is finite.
Let $\mathcal{N}(S)$ be the topology of $S$ induced by the metric $d_\mathbf{S}$,
that is more precisely,
$\fraktur{N}\in\mathcal{N}(S)$ if and only if $\,\fraktur{N}$ is a union of finite 
intersections of sets of the form
$
\fraktur{N}_{r}({\preceq})=
{\{{\preceq'}\in\TO(S)\mid d_\mathbf{S}({\preceq},{\preceq'})<2^{-r}\}}
$
where $r\in\NN_0$ and ${\preceq}\in\TO(S)$.
Then it holds ${\mathcal{N}(S)=\mathcal{U}(S)}$, 
in particular the topology $\mathcal{N}(S)$ is independent of the choice of $\mathbf{S}$,
and the topology $\mathcal{U}(S)$ is Hausdorff.
\end{thm}

\begin{proof}
Let $r\in\NN_0$ and ${\preceq} \in \TO(S)$.
We claim that $\fraktur{N}_{r}({\preceq})\in\mathcal{U}(S)$.
Indeed, let $\fraktur{U}=\bigcap_{(a,b)}\fraktur{U}_{(a,b)}$,
where the intersection is taken over all ordered pairs $(a,b)$ in $S_{r+1}\times S_{r+1}$
with $a \preceq b$.
Then ${\preceq}\in\fraktur{U}\in\mathcal{U}(S)$.
Hence ${\preceq'} \in \fraktur{N}_{r}({\preceq})$ if and only if
${\preceq'}\res_{S_{r+1}}\!=\,{\preceq}\res_{S_{r+1}}$,
and this is the case if and only if it holds $a \preceq' b \gdw a \preceq b$ for all 
$(a,b)\in S_{r+1} \times S_{r+1}$,
which is true if and only if ${\preceq'}\in\fraktur{U}$.
Thus $\fraktur{N}_r({\preceq})=\fraktur{U}$, and this shows that 
$\mathcal{N}(S)\subseteq\mathcal{U}(S)$.

On the other hand, let $(a,b)\in S\times S$ be any ordered
pair. 
We claim that the set $\fraktur{U}_{(a,b)}$ is open with respect to the metric $d_\mathbf{S}$.
Let ${\preceq}\in \fraktur{U}_{(a,b)}$, so that $a\preceq b$.
We find $r\in\NN_0$ such that $(a,b)\in S_{r+1}\times S_{r+1}$.
If ${\preceq'}\in \fraktur{N}_r({\preceq})$,
then ${\preceq'}\res_{S_{r+1}}\! = {\preceq}\res_{S_{r+1}}$,
in particular $a\preceq' b$, so that ${\preceq'}\in \fraktur{U}_{(a,b)}$,
thus $\fraktur{N}_r({\preceq})\subseteq \fraktur{U}_{(a,b)}$.
Hence $\fraktur{U}_{(a,b)}$ is open with respect to $\mathcal{N}(S)$,
and we conclude that $\mathcal{U}(S)\subseteq\mathcal{N}(S)$.
\end{proof}


\begin{cnv}
Henceforth, unless otherwise stated, whenever we refer to topological properties of $\TO(S)$,
we always intend that $\TO(S)$ is provided with the topology $\mathcal{U}(S)$.
Subsets of $\TO(S)$ are tacitly furnished with their relative to\-po\-lo\-gy with respect to
$\mathcal{U}(S)$.
Quotient sets of $\TO(S)$ by equivalence relations are  equipped with their quotient
topology with respect to $\mathcal{U}(S)$.
\end{cnv}

\begin{dfn}\label{axiom}
A \notion{filter} over a  set $X$ is a subset $\mathcal{F}$ of the power set $\mathcal{P}(X)$
of $X$
that enjoys the properties
(a)~$X\in\mathcal{F}$,
(b)~$\emptyset\notin\mathcal{F}$,
(c)~$A\subseteq B\subseteq X\und A\in \mathcal{F}\dann B\in \mathcal{F}$,
(d)~${A\in \mathcal{F}\und B\in \mathcal{F}\dann A\cap B\in \mathcal{F}}$.

An \notion{ultrafilter}  over $X$ is a filter $\mathcal{L}$ over $X$ that fulfills
the further property
(e)~$A\subseteq X\dann A\in \mathcal{L} \oder X\wo A\in \mathcal{L}$.
The disjunction in (e) is exclusive by (d) and (b).
Equivalently, an ultrafilter over $X$ is a maximal filter over $X$ with respect to inclusion.
\end{dfn}

\begin{thm}\label{compact}
$\TO(S)$ is compact.
\end{thm}

\begin{proof}
Suppose by contradiction that $\TO(S)$ is not compact.
Then we find an infinite index set $I$ and families $(a_i)_{i\in I}$ and $(b_i)_{i\in I}$
of elements $a_i,b_i\in S$ such that $(\mathfrak{U}_{(a_i,b_i)})_{i\in I}$ is a covering of
$\TO(S)$ which admits no finite subcovering.
Thus for each finite subset $s\subseteq I$ there exists ${\preceq_s}\in\TO(S)$ such that
${\preceq_s}\notin\bigcup_{i\in s}\mathfrak{U}_{(a_i,b_i)}$,
that is, for all ${i\in s}$ it holds~${a_i \succ_s b_i}$.

Let $I^\ast$ be the set of all non\-empty finite subsets of $I$.
For each $s\in I^\ast$
let us define 
$s^\ast = {\{t\in I^\ast\mid s\subseteq t\}}$.
Since $s\in s^\ast$ for all $s\in I^\ast$
and $s_1^\ast\cap s_2^\ast=(s_1\cup s_2)^\ast$ for all ${s_1,s_2\in I^\ast}$,
the set $\mathcal{S}={\{s^\ast\mid s\in I^\ast\}}$ has the finite intersection property,
that is to say,
any finite intersection of elements of $\mathcal{S}$ is non\-empty.
Therefore  ${\mathcal{F}=
\{Y\in\mathcal{P}(I^\ast)\mid
\exists\,n\in\NN\;\exists\,Z_1,\ldots,Z_n\in\mathcal{S}:Z_1\cap\ldots\cap Z_1\subseteq Y\}}$
is a filter over $I^\ast$ that extends $\mathcal{S}$.
Hence, by the Ultrafilter Lemma, which descends from Zorn's Lemma,
there exists an ultrafilter $\mathcal{L}$ over $I^\ast$ that extends $\mathcal{F}$,
so that $s^\ast\in \mathcal{L}$ for all $s\in I^\ast$.

We fix a family $(\preceq_s)_{s\in I^\ast}$ of total ordering ${\preceq_s}$ on $S$ as above
and define a binary relation ${\preceq}$ on $S$ by
$x\preceq y \Leftrightarrow \{s\in I^\ast\mid x\preceq_s y\}\in \mathcal{L}$.
By axioms (d) and (b) of \ref{axiom}, ${\preceq}$ is antisymmetric.
By axioms (d) and (c) of \ref{axiom}, ${\preceq}$ is transitive.
By axioms (e) and (c) of \ref{axiom}, ${\preceq}$ is total.
So ${\preceq}\in\TO(S)$.
On the other hand, by our choice of the orderings ${\preceq_s}$,
it holds $a_i\succ b_i$ for all $i\in I$,
thus ${\preceq}\notin\bigcup_{i\in I}\mathfrak{U}_{(a_i,b_i)}=\TO(S)$, a contradiction.
\end{proof}

\begin{dfn}
For each $a\in S$ let 
${{}\SO_{a}(S)=\{{\preceq}\in\TO(S) \mid \forall\,b\in S :a\preceq b\}}$,
the set of all \notion{$a$-founded orderings} on $S$.
\end{dfn}

\begin{cor}\label{closed}
For each $a\in S$ the set ${}\SO_{a}(S)$ is closed in $\TO(S)$,
and hence $\SO_{a}(S)$ is a compact subspace of ${}\TO(S)$.
\end{cor}

\begin{proof}
It holds 
$\SO_{a}(S)=\bigcap_{b\in S}\fraktur{U}_{(a,b)}$,
thus $\SO_a(S)$ is closed in $\TO(S)$ as each $\fraktur{U}_{(a,b)}$ is closed in $\TO(S)$
as observed in \ref{def 1}.
If $S$ is countable, then $\TO(S)$ is compact by \ref{compact}, and hence,
as a closed subset of a compact set, $\SO_a(S)$ equipped with its relative 
topology is compact.
\end{proof}

\section{Leading monomial ideals from total orderings}\label{chp1sec2}

\noindent
Let $t\in\NN$, let $K$ be a field,
and let $K[X]$ denote the commutative polynomial ring $K[X_1,\ldots,X_t]$.

\begin{rec+dfn}\label{Supp}
The countable set $M=\{X^\nu\mid\nu\in\NN_0^t\}$ of the
\notion{monomials of $K[X]$} is a basis of the $K$-module $K[X]$,
often referred to as the \notion{canonical $K$-basis of $K[X]$}.
We fix once for all this $K$-basis $M$ of $K[X]$.

Thus each  $p\in K[X]$ can be  written  in \notion{canonical form}
as $\sum_{\nu\in\supp(p)}\alpha_{\nu}X^\nu$
for a uniquely determined finite subset $\supp(p)$ of $\NN_0^t$ such that
$\alpha_{\nu}\in K\wo\{0\}$ for all $\nu\in\supp(p)$.
Notice that $\supp(p)=\emptyset$ if and only if $p=0$.

For each $p\in K[X]$ let us define the subset $\Supp(p)=\{X^\nu\mid \nu\in\supp(p)\}$ of $M$,
which we
call the \notion{support of $p$}.
Clearly, $\Supp(p)=\emptyset$ if and only if $p=0$.
We also put ${\Supp(E)=\bigcup_{e\in E}\Supp(e)}$ for each subset $E$ of $K[X]$. 

For each $p\in K[X]\wo\{0\}$ and each ${\leq}\in\TO(M)$ we denote by $\LT_{\leq}(p)$
the uniquely determined maximal element of $\Supp(p)$ with respect to ${\leq}$
and call $\LT_\leq(p)$ the \notion{leading monomial of $p$ with respect to~${\leq}$}.
In this situation, there exists a unique   $\alpha\in K\wo\{0\}$ 
such that either
$p-\alpha\LM_\leq(p)=0$ or $\LM_\leq(p-\alpha\LM_\leq(p))<\LM_\leq(p)$.
Such element $\alpha$ is denoted $\LC_\leq(p)$
and called the \notion{leading coefficient of $p$ with respect to~${\leq}$}.

For each $E\subseteq K[X]$ and each ${\leq}\in\TO(M)$ we denote by $\LT_\leq(E)$ 
the monomial ideal ${\langle\LT_{\leq}(e)\mid e\in E\wo\{0\}\rangle}$ of $K[X]$, and we call
$\LT_{\leq}(E)$ the \notion{leading monomial ideal of $E$ with respect to~${\leq}$}.

Finally,
let $\TTOO_\fraktur{S}(E)=\{\LT_{\leq}(E)\mid{\leq}\in\fraktur{S}\}$,
for $E\subseteq K[X]$ and $\fraktur{S}\subseteq\TO(M)$,
be the set of all \notion{leading monomial ideals of $E$ from $\fraktur{S}$}.
\end{rec+dfn}

\begin{rem}\label{cox242}
We shall, almost always tacitly, make use of the following well-known results,
see \cite[II.4.2 \& II.4.4]{Cox}.

Let $\Nu\subseteq\NN_0^t$.
Then a monomial $X^\upsilon$ of $K[X]$ lies in the ideal
$\langle X^\nu\mid\nu\in\Nu\rangle$ of $K[X]$
if and only if
there exists $\gamma\in\Nu$ such that $X^\gamma$ divides~$X^\upsilon$.

From this it follows that two monomials ideals are equal if and only if they contain the
same monomials.
\end{rem}

\begin{rem}\label{LT=LT2}
If $p\in K[X]$ and 
${\leq},{\leq'}\in\TO(M)$ are such that ${\leq}$ and ${\leq'}$ agree on $\Supp(p)$,
then clearly $\LT_{\leq}(p)=\LT_{\leq'}(p)$.

Hence, if ${\leq},{\leq'}\in\TO(M)$ and $F\subseteq K[X]$
are such that
${\leq}$ and ${\leq'}$ agree on $\Supp(F)$,
then
$\LM_{\leq}(F)
=\langle\LM_{\leq}(f)\mid f\in F\rangle
=\langle\LM_{\leq'}(f)\mid f\in F\rangle
=\LM_{\leq'}(F)$.

In this situation, if in addition we have $F\subseteq E\subseteq K[X]$
and $\LT_{\leq}(F)=\LT_{\leq}(E)$,
then clearly $\LT_{\leq}(E)\subseteq\LT_{\leq'}(E)$.
\end{rem}


\begin{dfn}\label{minimalizer}
Let $E\subseteq K[X]$ and let $\fraktur{S}\subseteq\TO(M)$.
We say that 
${\leq'}\in\fraktur{S}$ is a \notion{minimalizer of $E$ in $\fraktur{S}$}
if the 
condition
$\LT_{\leq}(E)\subseteq\LT_{\leq'}(E)$
already implies
$\LT_{\leq}(E)=\LT_{\leq'}(E)$
for all ${\leq}\in\fraktur{S}$, that is,
if $\LT_{\leq'}(E)$ is a minimal element of $\TTOO_\fraktur{S}(E)$
with respect to~${\subseteq}$.

We denote the set of all minimalizers of $E$ in $\fraktur{S}$ by $\MIN_E(\fraktur{S})$.
We write $\LLOO_\fraktur{S}(E)$ for the set
$\TTOO_{\min_E(\fraktur{S})}(E)=\{\LM_{\leq}(E)\mid{\leq}\in\MIN_E(\fraktur{S})\}$
of all \notion{minimal leading monomial ideals of $E$ from $\fraktur{S}$}.
\end{dfn}

\begin{lem}\label{LOE compatto}
Let $E\subseteq K[X]$ and $\fraktur{S}\subseteq\TO(M)$.
Then $\MIN_E(\fraktur{S})$ is a closed subset of~$\fraktur{S}$.
Hence, if $\fraktur{S}$ is closed in $\TO(M)$, then $\MIN_E(\fraktur{S})$ is compact.
\end{lem}

\begin{proof}
We  may choose a filtration $(S_i)_{i\in\NN_0}$ of $M$ consisting of finite
subsets $S_i$ of~$S$.
Let ${\leq}\in\fraktur{S}$ be any accumulation point of $\MIN_E(\fraktur{S})$.
Thus for each $r\in\NN_0$ there exists 
${{\leq_r}\in\MIN_E(\fraktur{S})\cap\fraktur{N}_r({\leq})\wo\{\leq\}}$.
Since $K[X]$ is noetherian, there exists a finite set $F\subseteq E$ such that
$\LM_{\leq}(E)=\LM_{\leq}(F)$. 
We can find $r\in\NN_0$ such that $\Supp(F)\subseteq S_{r+1}$.
We fix then ${\leq_r}\in\MIN_E(\fraktur{S})\cap\fraktur{N}_r({\leq})\wo\{\leq\}$.
Thus ${\leq}$ and ${\leq_r}$ agree on $S_{r+1}$ and in particular on $\Supp(F)$.
From \ref{LT=LT2} it follows $\LT_{\leq}(E)\subseteq\LT_{\leq_r}(E)$.
As ${\leq}\in\fraktur{S}$ and ${\leq_r}\in\MIN_E(\fraktur{S})$,
it follows $\LT_{\leq}(E)=\LT_{\leq_r}(E)$.
Hence $\LT_{\leq}(E)$ is a minimal element of $\TTOO_\fraktur{S}(E)$
with re\-spect to~${\subseteq}$, that is, ${\leq}\in\MIN_E(\fraktur{S})$.
Therefore $\MIN_E(\fraktur{S})$ contains all its accumulation points in~$\fraktur{S}$,
and hence $\MIN_E(\fraktur{S})$ is closed in~$\fraktur{S}$.
The statement about compactness follows now from~\ref{compact}.
\end{proof}

\begin{dfn}\label{equi}
Let $E\subseteq K[X]$ and $\fraktur{S}\subseteq\TO(M)$.
We define an equivalence relation ${\sim}_E$ on $\MIN_E(\fraktur{S})$ by
${\leq}\mathop{\sim_E}{\leq'} \Leftrightarrow \LT_{\leq}(E)=\LT_{\leq'}(E)$.
We also provide the set $\MIN_E(\fraktur{S})\mod {\sim_E}$ of the equivalence classes
of $\MIN_E(\fraktur{S})$ with respect to~${\sim_E}$ with its quotient topology.
\end{dfn}

\begin{rem}\label{min compact}
Let $E\subseteq K[X]$ and $\fraktur{S}\subseteq\TO(M)$.
By~\ref{LOE compatto}, $\MIN_E(\fraktur{S})\mod {\sim_E}$ is compact
whenever $\fraktur{S}$ is closed in $\TO(M)$.
\end{rem}

\begin{thm}\label{discreto}\label{finito}
Let $E\subseteq K[X]$ and $\fraktur{S}\subseteq\TO(M)$.
Then $\MIN_E(\fraktur{S})\mod {\sim_E}$ is discrete.
Hence, if $\fraktur{S}$ is closed in $\TO(M)$, then $\MIN_E(\fraktur{S})\mod {\sim_E}$
is finite.
\end{thm}

\begin{proof}
Let $\pi_E:\MIN_E(\fraktur{S})\rightarrow\MIN_E(\fraktur{S})\mod {\sim_E}$
be the natural projection that maps each ${\leq}$ to its equi\-valence class $[\leq]$
with respect to~${\sim_E}$.
Let ${\leq}\in\MIN_E(\fraktur{S})$.
It is enough to show that $\{[\leq]\}$ is open in $\MIN_E(\fraktur{S})\mod {\sim_E}$.
Put $\mathfrak{U}=\pi_E^{-1}([\leq])$.
By definition, $\{[\leq]\}$ is open in $\MIN_E(\fraktur{S})\mod {\sim_E}$
if and only if $\mathfrak{U}$ is open in $\MIN_E(\fraktur{S})$.
 
We may assume that $\mathfrak{U}\ne\emptyset$.
Let ${\leq'} \in \mathfrak{U}$.
We aim to find an open subset $\mathfrak{V}$ of $\MIN_E(\fraktur{S})$ such that
${{\leq'} \in \mathfrak{V} \subseteq \mathfrak{U}}$.
As $K[X]$ is noetherian,
there exists a finite subset $F$ of $E$ with
${\LT_{\leq'}(F)=\LT_{\leq'}(E)}$.
Let $(S_i)_{i\in\NN_0}$ be a filtration of $M$ by finite sets $S_i$.
As the set $\Supp(F)$ is finite, we find $r\in\NN_0$ such that $\Supp(F)\subseteq S_{r+1}$.
Put $\mathfrak{V}=\mathfrak{N}_r(\leq')\cap\MIN_E(\fraktur{S})$.
Of course, $\mathfrak{V}$ is open in $\MIN_E(\fraktur{S})$ and ${\leq'} \in \mathfrak{V}$.

We claim that $\mathfrak{V} \subseteq \mathfrak{U}$.
Let ${\leq''} \in \mathfrak{V}$.
Then ${\leq'}$ and ${\leq''}$ agree on $S_{r+1}$ and hence on $\Supp(F)$.
It follows $\LT_{\leq'}(E)\subseteq\LT_{\leq''}(E)$, as we have already observed 
in~\ref{LT=LT2}.
Because ${\leq''}\in\MIN_E(\mathfrak{S})$ and ${\leq'}\in\mathfrak{S}$,
we obtain $\LT_{\leq'}(E)=\LT_{\leq''}(E)$.
Thus $[\leq'']=[\leq']=[\leq]$, that is, ${\leq''} \in \mathfrak{U}$.

Hence $\mathfrak{V} \subseteq \mathfrak{U}$,
so $\mathfrak{U}$ is open in $\MIN_E(\fraktur{S})$.
We have proved that $\MIN_E(\fraktur{S})\mod {\sim_E}$ is discrete.
If $\fraktur{S}$ is closed in $\TO(M)$,
then $\MIN_E(\fraktur{S})\mod {\sim_E}$ is also compact by \ref{min compact},
and hence finite.
\end{proof}

\begin{cor}\label{fini}
For each $E\subseteq K[X]$ and each closed $\fraktur{S}\subseteq\TO(M)$
the set $\LLOO_\fraktur{S}(E)$ is finite, that is,
there exist at most finitely many distinct minimal leading monomial ideals of $E$
from~$\fraktur{S}$.
\end{cor}

\begin{proof}
The statement follows from \ref{finito} as clearly there exists a bijection between
the sets $\LLOO_\fraktur{S}(E)$ and $\MIN_E(\fraktur{S}) \mod {\sim_E}$ given by
$\LM_\leq(E)\mapsto [\leq]$ for all ${\leq}\in\MIN_E(\fraktur{S})$.
\end{proof}

\section{Leading monomial ideals from degree orderings}\label{chp1sec2'}

\noindent
We keep the notation of the previous section.

\begin{dfn}
For all $s\in\NN_0$ we denote by $K[X]_{\le s}$ the $K$-submodule of $K[X]$ 
of finite length
consisting of all polynomials of total degree less than or equal to~$s$.
Given any subset $E$ of $K[X]$, we put
$E_{\le s}=K[X]_{\le s}\cap E$ for all~$s\in\NN_0$.

Let $I$ be an ideal of $K[X]$.
Then $I_{\le s}$ is a $K$-submodule of $K[X]_{\le s}$.
Therefore, as in $\text{\cite[IX.3.2]{Cox}}$, we may define the
\notion{Hilbert function $\HF_I:\NN_0\rightarrow\NN_0$ of $I$}
by the assignment ${s\mapsto\len_K K[X]_{\le s}\mod I_{\le s}}$.

By \cite[IX.3.3(a)]{Cox}, if $I$ is a monomial ideal,
then $\HF_I(s)$ equals the cardinality of the set $M_{\le s}\wo I_{\le s}$.

Moreover, by \cite[IX.2.4 \& IX.3.3(b)]{Cox},
there exists a uniquely determined univariate polynomial $\HP_I$ with rational coefficients
and at most of degree $t$ with the property that $\HP_I(s)=\HF_I(s)$ for $s\gg 0$,
the \notion{Hilbert polynomial of $I$}.

We may thus define $\ind(I)=\min{}\{s_0\in\NN_0\mid\forall\,s\ge s_0:\HF_I(s)=\HP_I(s)\}
\in\NN_0$,
the \notion{index of regularity of $I$}.
\end{dfn}

\begin{lem}\label{regula}
If $I$ and $J$ are monomial ideals of $K[X]$ such that $I\subseteq J$,
then $\ind(I)\ge\ind(J)$.
\end{lem}

\begin{proof}
This follows from \cite[IX.2.5 \& IX.3.3]{Cox}.
See also the proof of \cite[IX.2.6]{Cox}.
\end{proof}

\begin{lem}\label{hilbert}
If $I$ and $J$ are monomial ideals of $K[X]$ with 
$I\subseteq J$ and ${}\HF_I=\HF_J$,
then $I=J$.
\end{lem}

\begin{proof}
If there existed a monomial $m\in J\wo I$,
then with $s=\deg(m)$ it would hold $I_{\le s}\subsetneq J_{\le s}$,
thus $\HF_I(s)=\vert M_{\le s}\wo I_{\le s}\vert > \vert M_{\le s}\wo J_{\le s}\vert=\HF_J(s)$,
a contradiction.
Hence $I\cap M=J\cap M$, whence $I=J$ as these are monomial ideals,
see also~\ref{cox242}.
\end{proof}

\begin{dfn}
One clearly has ${\deg(\LT_\leq(p))\le\deg(p)}$
for all ${\leq}\in\TO(M)$ and all ${p\in K[X]\wo\{0\}}$,
where $\deg({-})$ denotes the total degree function on $K[X]$.
A \notion{degree ordering on $M$} or \notion{of $K[X]$} is a total ordering
${\leq}$ on $M$ such that
it holds 
$\deg(\LT_\leq(p))=\deg(p)$
for all $p\in K[X]\wo\{0\}$.
The set of all degree orderings on $M$ is denoted $\DO(M)$.
\end{dfn}

\begin{exa}\label{DO exa}
For each ${\leq}\in\TO(M)$ the binary relation ${\leq_{\deg}}$ on $M$ defined by
\begin{equation*}
m \leq_{\deg} m' \Leftrightarrow \deg(m)<\deg(m')\vee (\deg(m)=\deg(m')\wedge m\le m')
\end{equation*}
is a degree ordering of $K[X]$. 
\end{exa}

\begin{pro}\label{dco1}
It holds $\DO(M)\subseteq\SO_1(M)$.
\end{pro}

\begin{proof}
Let ${\leq}\in\DO(M)$.
Suppose ${\leq}\notin\SO_1(M)$.
Then there exists $m\in M$ such that $1\not\leq m$.
So $m<1$ by totality.
It follows $\LT_\leq(m+1)=1$, thus $\deg(\LT_{\leq}(m+1))=0$.
But $m$ is a monomial different than $1$, hence $\deg(m+1)>0$,
a contradiction. 
\end{proof}

\begin{rec}
Let $S$ be a set.
We recall that a \notion{partial ordering on $S$} is a reflexive, transitive, and
antisymmetric binary relation on $S$, and that a partial ordering ${\preceq}$ on $S$ is said a
\notion{well-ordering on $S$} if each non\-empty subset $T$ of $S$ admits a minimal element
with respect to ${\preceq}$,
that is, for each $T\subseteq S$ with $T\ne\emptyset$ there exists $t'\in T$ such that for each
$t\in T$ it holds the implication $t\preceq t'\Rightarrow t=t'$.

If ${\preceq}$ is a total ordering of $S$, then ${\preceq}$ is a well-ordering on $S$
precisely when each non\-empty subset $T$ of $S$ admits a minimum, that is,
for each $T\subseteq S$ with $T\ne\emptyset$ there exists $t'\in T$ such that for each
$t\in T$ it holds ${t'}\!\preceq t$.
\end{rec}

\begin{ntn}
For each set $S$ we denote by $\WO(S)$ the set of all total orderings on $S$ that are also
well-orderings on $S$.
\end{ntn}

\begin{pro}\label{dco2}
It holds $\DO(M)\subseteq\WO(M)$.
\end{pro}

\begin{proof}
Let ${\leq}\in\DO(M)$.
Let $\emptyset\ne T\subseteq M$.
Suppose that there exists no minimum in $T$ with respect to ${\leq}$.
Let $t_0\in T$.
We find $t_1\in T$ such that $t_1<t_0$, and then find $t_2\in T$ such that $t_2<t_1$,
and then$\ldots$
Thus there exists in $T$ an infinite strictly descending chain $\ldots < t_2 < t_1 < t_0$.

For each $k\in\NN_0$ it holds $\deg(t_k)\ge\deg(t_{k+1})$.
Indeed, let $k\in\NN_0$ and consider the polynomial $t_k+t_{k+1}$.
We have $\LT_{\leq}(t_k+t_{k+1})=t_k$ as $t_{k}>t_{k+1}$.
Since ${\leq}\in\DO(M)$, it follows $\deg(t_k+t_{k+1})=\deg(t_k)$.
Hence $\deg(t_{k})\ge\deg(t_{k+1})$.

Therefore we can write $\ldots\le\deg(t_2)\le\deg(t_1)\le\deg(t_0)$.
Now, for each $d\in\NN_0$ there exist only finitely many distinct monomials of degree $d$.
Hence we can find a sequence $(k_i)_{i\in\NN_0}$ of integers $k_i$ with $k_0=0$ and
$k_i<k_{i+1}$ with the property that the strict descending chain
$\ldots<\deg(t_{k_2})<\deg(t_{k_1})<\deg(t_{k_0})$
in $\NN_0$ is infinite, and this is absurd.
%
\end{proof}

\begin{lem}\label{DO(M) compact}
$\DO(M)$ is a closed subset of ${}\TO(M)$ and hence compact.
\end{lem}

\begin{proof}
Let $(S_i)_{i\in\NN_0}$ be a filtration of $M$ consisting of finite sets $S_i$.
Let ${{\leq}\in\TO(M)}$ be an accumulation point of $\DO(M)$.
For each $r\in\NN_0$ 
we find ${\leq}_r$ in $\DO(M)\cap\mathfrak{N}_r({\leq})$ with ${\leq}_r \ne {\leq}$,
so that ${\leq}$ and ${\leq_r}$ agree on $S_{r+1}$.
Let $p\in K[X]\wo\{0\}$.
We  find ${r\in\NN_0}$ such that ${\Supp(p)\subseteq S_{r+1}}$.
We  choose ${\leq_r}$ as above, and so $\LT_{\leq}(p)=\LT_{\leq_r}(p)$,
thus $\deg(\LT_{\leq}(p))=\deg(\LT_{\leq_r}(p))=\deg(p)$ as ${\leq_r}$ is a degree ordering.
Hence ${\leq}\in\DO(M)$.
Therefore $\DO(M)$ contains all its accumulation points in $\TO(M)$ and so 
is closed in $\TO(M)$.
Since $\TO(M)$ is compact by \ref{compact}, it follows that $\DO(M)$ is compact.
\end{proof}

\begin{dfn}\label{equiequi}
Let $E\subseteq K[X]$ and $\fraktur{S}\subseteq\DO(M)$.
Analogously as in \ref{equi}, we define an equivalence relation ${\sim_E}$ on $\fraktur{S}$
by ${\leq}\sim_E {\leq'} \Leftrightarrow \LM_{\leq}(E)=\LM_{\leq'}(E)$.
We also provide the set $\fraktur{S}$ with its relative topology and the set
$\fraktur{S}\mod {\sim_E}$ of the equivalence classes
of $\fraktur{S}$ with respect to~${\sim_E}$ with its quotient topology.
\end{dfn}

\begin{rem}\label{S compact}
Let $E\subseteq K[X]$ and $\fraktur{S}\subseteq\DO(M)$.
From~\ref{DO(M) compact} it follows that $\fraktur{S}\mod {\sim_E}$ is compact
whenever $\fraktur{S}$ is closed in $\DO(M)$.
By \ref{DO(M) compact} it is also clear that $\mathfrak{S}$ is closed in $\DO(M)$
if and only if $\mathfrak{S}$ is closed in $\TO(M)$.
\end{rem}

\begin{lem}\label{llll}
Let $E\subseteq K[X]$ and ${\leq}\in\DO(M)$.
There exists an open neighbourhood $\mathfrak{U}$ of $\,{\leq}\,$ in $\DO(M)$ such that
$\LM_{\leq}(E)=\LM_{\leq'}(E)$ for all ${\leq'}\in\mathfrak{U}$.
\end{lem}

\begin{proof}
Fix a filtration $(S_i)_{i\in\NN_0}$ of $M$ by finite sets $S_i$.
As $K[X]$ is noetherian,
there exists a finite subset $F$ of $E$ such that
${\LT_{\leq}(F)=\LT_{\leq}(E)}$.
Put $s_0=\ind(\LM_{\leq}(E))$
and recall that $t$ is the number of indeterminates of our polynomial ring $K[X]$.
As $M=\bigcup_{i\in\NN_0}S_i$ and as the sets $\Supp(F)$ and $M_{\le s_0+t}$ are finite,
we find $r\in\NN_0$ such that $\Supp(F)\cup M_{\le s_0+t}\subseteq S_{r+1}$.
Trivially $\mathfrak{U}=\mathfrak{N}_{r}({\leq})\cap\DO(M)$ is open in $\DO(M)$,
and clearly ${\leq}\in\mathfrak{U}$.

Let ${\leq'}\in\mathfrak{U}$.
Since ${\leq}$ and ${\leq'}$ agree on $S_{r+1}$ and hence on $\Supp(F)$,
by \ref{LT=LT2} we get (a)~$\LT_{\leq}(E)\subseteq\LT_{\leq'}(E)$.
Similarly, ${\leq}$ and ${\leq'}$ agree on $M_{\le s_0+t}$,
and because ${\leq}$ and ${\leq'}$ are degree orderings,
we obtain $\LM_{\leq}(E)_{\le s}=\LM_{\leq'}(E)_{\le s}$ for $0\le s\le s_0+t$,
and hence $\vert M_{\le s}\wo\LM_{\leq}(E)_{\le s}\vert
=\vert M_{\le s}\wo\LM_{\leq'}(E)_{\le s}\vert$  for $0\le s\le s_0+t$,
and therefore we have
(b)~$\HF_{\LM_{\leq}(E)}(s)=\HF_{\LM_{\leq'}(E)}(s)$ for ${0\le s\le s_0+t}$.
By (a) and \ref{regula} it holds
$\ind(\LM_{\leq}(E))\ge\ind(\LM_{\leq'}(E))$.
It follows $\HP_{\LM_{\leq}(E)}(s)=\HP_{\LM_{\leq'}(E)}(s)$ for ${s_0\le s\le s_0+t}$.
As the polynomials $\HP_{\LM_{\leq}(E)}$ and $\HP_{\LM_{\leq'}(E)}$ have at most degree $t$
and as they agree on $t+1$ points,
it follows (c)~$\HP_{\LM_{\leq}(E)}=\HP_{\LM_{\leq'}(E)}$.
By (b) and (c) we get $\HF_{\LM_{\leq}(E)}=\HF_{\LM_{\leq'}(E)}$.
Hence, by \ref{hilbert},  $\LM_{\leq}(E)=\LM_{\leq'}(E)$.
\end{proof}

\begin{thm}\label{endlich}
Let $E\subseteq K[X]$ and $\mathfrak{S}\subseteq\DO(M)$.
Then $\mathfrak{S} \mod {\sim}_E$ is discrete.
Hence, if ${}\mathfrak{S}$ is closed in $\DO(M)$, then $\mathfrak{S} \mod {\sim}_E$ is finite.
\end{thm}

\begin{proof}
Let $\pi_E:\mathfrak{S}\rightarrow\mathfrak{S}\mod {\sim_E}$
be the natural projection that maps each ${\leq}$ to its equi\-valence class $[\leq]$
with respect to~${\sim_E}$.
Let ${{\leq}\in\mathfrak{S}}$.
It is enough to show that $\{[\leq]\}$ is open in $\mathfrak{S} \mod {\sim_E}$.
Put $\mathfrak{U}=\pi_E^{-1}([\leq])$.
By definition, $\{[\leq]\}$ is open in $\mathfrak{S} \mod {\sim_E}$
if and only if $\mathfrak{U}$ is open in~$\mathfrak{S}$.
 
We may assume that $\mathfrak{U}\ne\emptyset$.
Let ${\leq'}\in\mathfrak{U}$.
We aim to find an open subset $\mathfrak{W}$ of $\mathfrak{S}$ such that
${\leq'}\in\mathfrak{W}\subseteq\mathfrak{U}$.
By \ref{llll}, we find an open subset $\mathfrak{V}$ of $\DO(M)$ with ${\leq'}\in\mathfrak{V}$
such that for all
${\preceq''}\in\mathfrak{V}$ it holds $[\preceq'']=[\preceq']=[\preceq]$.
Thus, putting ${\mathfrak{W}=\mathfrak{V}\cap\mathfrak{S}}$, we have that
$\mathfrak{W}$ is open in $\mathfrak{S}$ and ${\preceq'}\in\mathfrak{W}\subseteq\mathfrak{U}$.

Therefore $\mathfrak{U}$ is open in~$\fraktur{S}$.
We have proved that $\fraktur{S}\mod {\sim_E}$ is discrete.
If $\fraktur{S}$ is closed in $\DO(M)$, then $\fraktur{S}$ and thus $\fraktur{S}\mod {\sim_E}$
are also compact by \ref{DO(M) compact}, and hence $\fraktur{S}\mod {\sim_E}$ is finite.
\end{proof}

\begin{cor}\label{endlich2}
For each $E\subseteq K[X]$ and each $\fraktur{S}\subseteq\DO(M)$
the set $\TTOO_\fraktur{S}(E)$ is finite, that is,
there exist at most finitely many distinct leading monomial ideals
of $E$ from~$\fraktur{S}$.
\end{cor}

\begin{proof}
Let $E\subseteq K[X]$.
By \ref{endlich}, $\DO(M) \mod {\sim_E}$ is finite.
We have a bijection between
the sets $\TTOO_{\DO(M)}(E)$ and $\DO(M) \mod {\sim_E}$ given by
$\LM_\leq(E)\mapsto [\leq]$ for all ${\leq}\in\DO(M)$,
thus $\TTOO_{\DO(M)}(E)$ is finite.
Now, if $\fraktur{S}\subseteq\DO(M)$, then
$\TTOO_{\fraktur{S}}(E)\subseteq\TTOO_{\DO(M)}(E)$.
\end{proof}

\section{Action of $K$-module isomorphisms}

\noindent
We keep the notation of the previous section.
Further, let $V$ be a $K$-module such that there exists a $K$-module isomorphism $\slPhi$
of $V$ in $K[X]$, and put $N=\slPhi^{-1}(M)$, so that $N$ is a countable $K$-basis of $V$.
Sometimes we denote the inverse of $\slPhi$ by~$\slPsi$.

\begin{rem}\label{phi}
We have a map $\phi:\TO(N)\rightarrow\TO(M)$ such that
for any given ${\preceq}\in\TO(N)$ it holds
$\slPhi(n)\mathop{\phi({\preceq})}\slPhi(n') \Leftrightarrow n\preceq n'$ for all $n,n'\in N$.

Indeed, fixed any ${\preceq}\in\TO(N)$, simply define
$m \mathop{\phi({\preceq})} m' \Leftrightarrow \slPhi^{-1}(m)\preceq \slPhi^{-1}(m')$
for all $m,m'\in M$.
Then $\phi({\preceq})$ is uniquely determined by ${\preceq}$
as $\Phi^{-1}$ is surjective,
and $\phi({\preceq})$ is total and hence reflexive and is transitive
as ${\preceq}$ is.
The antisymmetry of $\phi({\preceq})$ follows immediately from the injectivity of~$\slPhi^{-1}$.

In a similar way, there exists a map $\psi:\TO(M)\rightarrow\TO(N)$ such that
for any given ${\leq}\in\TO(M)$ it holds
$\slPsi(m)\mathop{\psi({\leq})}\slPsi(m') \Leftrightarrow m\leq m'$ for all $m,m'\in M$.

The maps $\phi$ and $\psi$ are inverse of each other, thus they are isomorphisms of sets.
Indeed, they are more, as the following theorem asserts.
\end{rem}

\begin{thm}\label{homeo}
The bijection $\phi$
of \ref{phi}
is a homeomorphism of ${}\TO(N)\!$ in $\TO(M)$.
\end{thm}

\begin{proof}
We only have to show that $\phi$ is continuous and open.
Since $\phi$ is bijective, it is enough to check this for one choice of subbases of $\TO(N)$ 
and $\TO(M)$.

For each $(n,n')\in N\times N$
one has $\phi(\fraktur{U}_{(n,n')})=\fraktur{U}_{(\slPhi(n),\slPhi(n'))}$,
thus $\phi$ is open.
For each $(m,m')\in M\times M$
it holds $\phi^{-1}(\fraktur{U}_{(m,m')})
=\fraktur{U}_{(\slPhi^{-1}(m),\slPhi^{-1}(m'))}$,
hence $\phi$ is continuous.
\end{proof}

\newcommand{\HHH}{H}
\newcommand{\hhh}{h}

\begin{dfn+rem}\label{Supp2}
Each  $v\in V$ can be  written  in \notion{canonical form}
as a sum $\sum_{n\in\Supp(v)}\alpha_{n}n$
for a uniquely determined finite subset $\Supp(v)$ of $N$ such that
$\alpha_{n}\in K\wo\{0\}$ for all $n\in\Supp(v)$.
We call $\Supp(v)$ the \notion{support of $v$}.
For each subset $\HHH$ of~$V$ let $\Supp(\HHH)=\bigcup_{\hhh\in \HHH}\Supp(\hhh)$.

In the notation of \ref{Supp}, one has
$\Supp(\slPhi(v))=\slPhi(\Supp(v))$ for all $v\in V$,
and hence
$\Supp(\slPhi(\HHH))=\slPhi(\Supp(\HHH))$ for all $\HHH\subseteq V$.
Conversely, 
$\Supp(\slPsi(p))=\slPsi(\Supp(p))$ for all $p\in K[X]$,
and hence
$\Supp(\slPsi(E))=\slPsi(\Supp(E))$ for all $E\subseteq K[X]$.

Given any ${\preceq}\in\TO(N)$, for each $v\in V\wo\{0\}$ we denote by $\lt_{\preceq}(v)$
the uniquely determined maximal element of $\Supp(v)$ with respect to ${\preceq}$.

In the notation of \ref{phi},
one has $\LT_{\phi({\preceq})}(\slPhi(v))=\slPhi(\lt_\preceq(v))$ 
for all ${v\in V\wo\{0\}}$.
For each $v\in V\wo\{0\}$ we write $\LT_{\preceq}(v)$ for $\LT_{\phi({\preceq})}(\slPhi(v))$,
and with abuse of language we call $\LT_\preceq(v)$ the \notion{leading monomial of $v$
with respect to~${\preceq}$}.
In this situation, we denote $\LC_{\phi({\preceq})}(\slPhi(v))$ by $\LC_\preceq(v)$
or $\lc_\preceq(v)$,
and with abuse of language we call $\LC_\preceq(v)$ alias $\lc_\preceq(v)$
the \notion{leading coefficient of $v$ with respect to~${\preceq}$}.
Observe that either $v-\lc_\preceq(v)\lm_\preceq(v)=0$ or 
$\lm_\preceq(v-\lc_\preceq(v)\lm_\preceq(v)) \prec \lm_\preceq(v)$.

For each ${\preceq}\in\TO(N)$ and each $\HHH\subseteq V$ we denote 
by $\LT_\preceq(\HHH)$
the monomial ideal $\langle\LT_{\preceq}(\hhh)\mid \hhh\in \HHH\wo\{0\}\rangle$ of $K[X]$,
and again with abuse of language we call
$\LT_{\preceq}(\HHH)$ the \notion{leading monomial ideal of $\HHH$ with respect to~${\preceq}$}.

Further, for each $\HHH\subseteq V$ and each $\fraktur{T}\subseteq\TO(N)$
let $\TTOO_{\fraktur{T}}(\HHH)=\{\LT_{\preceq}(\HHH)\mid{\preceq}\in\fraktur{T}\}$
be the set of all \notion{leading monomial ideals of $H$ from $\fraktur{T}$}.

Similarly as in \ref{minimalizer}, given $\HHH\subseteq V$ and $\fraktur{T}\subseteq\TO(N)$,
we say that ${\preceq}\in\TO(N)$ is a \notion{minimalizer of $\HHH$ in $\fraktur{T}$}
if $\LM_\preceq(\HHH)$ is a minimal element of $\TTOO_{\fraktur{T}}(\HHH)$
with respect to~${\subseteq}$.
 
We denote the set of all minimalizers of $\HHH$ in $\fraktur{T}$ by $\MIN_\HHH(\fraktur{T})$.
We write $\LLOO_\fraktur{T}(\HHH)$ for the set 
$\TTOO_{\min_H(\fraktur{T})}(H)=\{\LM_{\preceq}(\HHH)\mid{\preceq}\in\MIN_\HHH(\fraktur{T})\}$ 
of all \notion{minimal leading monomial ideals of $\HHH$ from $\fraktur{T}$}.
\end{dfn+rem}

\begin{rem}
Let $\fraktur{T}\subseteq\TO(N)$ and $H\subseteq V$.
The homeomorphism $\smash{\phi\res_{\fraktur{T}}:\fraktur{T}\rightarrow\phi(\fraktur{T})}$
induces a homeomorphism
$\smash{\overline{\phi\res_{\fraktur{T}}} : 
\fraktur{T}/{\sim}_H\rightarrow\phi(\fraktur{T})/{\sim}_{\slPhi(H)}}$
with
$\pi_{\Phi(H)}\circ\phi\res_{\fraktur{T}}=\overline{\phi\res_{\fraktur{T}}}\circ\pi_{H}$,
where ${\sim}_{H}$ is the equivalence relation on $\fraktur{T}$ given by
${\preceq}\sim_H{\preceq'}$ if and only if $\LM_\preceq(H)=\LM_{\preceq'}(H)$,
and ${\sim}_{\slPhi(H)}$ is the equivalence relation on $\phi(\fraktur{T})$
defined as in \ref{equiequi},
and $\pi_H$ and $\pi_{\slPhi(H)}$ are the respective natural projections.
\end{rem}

\begin{rem}\label{parapapunzi}
Given any $\HHH\subseteq V$ and $\fraktur{T}\subseteq\TO(N)$,
it follows immediately from the definitions that
$\LT_\preceq(\HHH)=\LT_{\phi({\preceq})}(\slPhi(H))$ for all ${\preceq}\in\fraktur{T}$.
Conversely, given any $E\subseteq K[X]$ and $\fraktur{S}\subseteq\TO(M)$, one has
$\LT_\leq(E)=\LT_{\psi({\leq})}(\slPsi(E))$ for all ${\leq}\in\fraktur{S}$.
It immediately follows that $\TTOO_\fraktur{T}(\HHH)=\TTOO_{\phi(\fraktur{T})}(\slPhi(\HHH))$
and
$\TTOO_\fraktur{S}(E)=\TTOO_{\psi(\fraktur{S})}(\slPsi(E))$,
and
even 
that
$\LLOO_\fraktur{T}(\HHH)=\LLOO_{\phi(\fraktur{T})}(\slPhi(\HHH))$
and
$\LLOO_\fraktur{S}(E)=\LLOO_{\psi(\fraktur{S})}(\slPsi(E))$.
\end{rem}

\begin{thm}\label{fini2}
Let $H\subseteq V$ and let $\mathfrak{T}\subseteq\TO(N)$ be closed.
Then $\LLOO_\fraktur{T}(H)$ is finite,
that is, there exist at most finitely many distinct minimal leading monomial ideals of $H$
from~$\fraktur{T}$.
\end{thm}

\begin{proof}
Clear by \ref{parapapunzi}, \ref{homeo}, and \ref{fini}.
\end{proof}

\begin{dfn}
We put $\DO(N)=\phi^{-1}(\DO(M))$,
and call $\DO(N)$ the set of all \notion{degree orderings on $N$}.
\end{dfn}

\begin{rem}\label{WO M N}
$\SO_{\slPhi^{-1}(1)}(N)=\phi^{-1}(\SO_1(M))$ and $\WO(N)=\phi^{-1}(\WO(M))$.
Hence $\DO(N)\subseteq\SO_{\slPhi^{-1}(1)}(N)\cap\WO(N)$ by \ref{dco1} and \ref{dco2}.
Moreover, by \ref{homeo} and \ref{DO(M) compact}, $\DO(N)$ is closed in $\TO(N)$ and compact.
\end{rem}

\begin{thm}\label{fini3b}
For each $H\subseteq V$ and each $\mathfrak{T}\subseteq\DO(N)$ the set $\TTOO_\fraktur{T}(H)$
is finite, that is, there exist at most finitely many distinct leading monomial ideals of $H$
from~$\fraktur{T}$.
\end{thm}

\begin{proof}
Clear by \ref{parapapunzi} and \ref{endlich2}.
\end{proof}

\section{$\mathfrak{T}$-multiplicative algebras of countable type}\label{mact}

\noindent
We keep the notation of the previous section.

\begin{dfn}
An algebra of \notion{countable type} is  a quadruple $A_K^{t,\slPhi}=(A,K,t,\slPhi)$ 
consisting of an associative, not necessarily commutative algebra $A$ over a field $K$,
a non\-negative integer $t$,
and a $K$-module isomorphism $\slPhi$ of $A$ in $K[X_1,\ldots,X_t]$.

If $A_K^{t,\slPhi}$ is an algebra of countable type and if $M$ is the canonical $K$-basis
of  $K[X_1,\ldots,X_t]$
consisting of all monomials $X^\nu$, $\nu\in\NN_0^t$,
then $N=\slPhi^{-1}(M)$ is a countable
$K$-basis of $A$, which we call the \notion{canonical basis of $A_K^{t,\slPhi}$}.

Given any subset $\mathfrak{T}$ of the set $\TO(N)$ of all total orderings on $N$,
we say that $A_K^{t,\slPhi}$ or simply $A$ is \notion{multiplicative on $\mathfrak{T}$} or
\notion{$\mathfrak{T}\text{-}$multiplicative} if $A$ is a domain and 
in the notation of \ref{Supp2} it holds $\LT_\preceq(ab)=\LT_\preceq(a)\LT_\preceq(b)$
for all $a,b\in A\wo\{0\}$ and all ${\preceq}\in\mathfrak{T}$.
\end{dfn}

\noindent
Henceforth in this section, let $A_K^{t,\slPhi}$ be an algebra of countable type.
We  write $K[X]$ for $K[X_1,\ldots,X_t]$  
and fix the canonical $K$-bases $M$ and $N$ of $K[X]$ and $A_K^{t,\slPhi}$\!, respectively.
Now we may make use of the notation introduced in~\ref{Supp2}.
And yet another$\ldots$ Macaulay Basis Theorem, that is,
a slight generalization of a classical result.

\begin{thm}\label{macaulay}
Let ${\preceq}\in\WO(N)$,
assume that $A_K^{t,\slPhi}$ is  multiplicative on $\{\preceq\}$,
let $L$ be a left ideal of $A$,
put $B=M\wo \LM_\preceq(L)$,
and let $\overline{\phantom{\bullet}} : K[X]\rightarrow K[X]\mod \slPhi(L)$
be the residue class epimorphism of $K$-modules.
Then the image $\overline{B}$ of $B$ under $\overline{\phantom{\bullet}}$ is a $K$-basis
of $K[X]\mod \slPhi(L)$.
\end{thm}

\begin{proof}
We first show that $\overline{B}$ generates $K[X]\mod \slPhi(L)$ over $K$.
Suppose it is not the case.
Let $\overline{W}=\sum_{b\in B}K\overline{b}$.
Then the set $P=\{p\in K[X]\wo\{0\}\mid\overline{p}\notin\overline{W}\}$ is non\-empty.
Thus, with ${\leq}=\phi({\preceq})$,
the subset $Q=\{\LM_{\leq}(p)\mid p\in P\}$ of $M$ is non\-empty.
As ${\phi(\preceq})\in\WO(M)$, see \ref{WO M N}, we may choose $p\in P$
such that $\LM_{\leq}(p)$ is minimal in $Q$ with respect to ${\leq}$.
It holds $\overline{\Supp(p)\wo\{\LM_{\leq}(p)\}}\subseteq\overline{W}$.
Indeed, if there existed $m\in\Supp(p)\wo\{\LM_{\leq}(p)\}$
such that ${\overline{m}\notin\overline{W}}$,
then we would have $m\in P$ and hence $m=\LM_{\leq}(m)\in Q$,
and this would contradict the minimality of $\LM_{\leq}(p)$ as clearly $m < \LM_{\leq}(p)$.
It follows $\overline{\LM_{\leq}(p)}\notin\overline{W}$
as otherwise we would have $\overline{\Supp(p)}\subseteq\overline{W}$ and hence
the contradiction $\overline{p}\in\overline{W}$.
Therefore ${\LM_{\leq}(p)\in\LM_\preceq(L)}$
as otherwise we would have $\LM_{\leq}(p)\in B$ and
hence the contradiction $\overline{\LM_{\leq}(p)}\in\overline{B}\subseteq\overline{W}$.
Thus we find $x\in L\wo\{0\}$ such that
$\LM_\preceq(x)\mid\LM_{\leq}(p)$, see \ref{cox242}.
So we find $n\in N$ with 
${\LM_{\leq}(p)=\slPhi(n)\LM_\preceq(x)=\LM_\preceq(n)\LM_{\preceq}(x)=\LM_{\preceq}(nx)}$,
where this last equality holds by  multiplicativity of $A_K^{t,\slPhi}$ on~$\{{\preceq}\}$.
With $q=\LC_\leq(p)\LC_\leq(\slPhi(nx))^{-1}\slPhi(nx)$ we obtain $q\in\slPhi(L)$
as $L$ is a left ideal and $\slPhi(L)$ is a $K$-module,
and of course we have $\LM_{\leq}(p)=\LM_{\leq}(q)$ and $\LC_\leq(p)=\LC_\leq(q)$.
Now we consider $r=p-q$.
It holds $\overline{r}=\overline{p}$.
Thus $\overline{r}\notin\overline{W}$.
But then in particular $r\ne 0$,
and hence clearly $\LM_{\leq}(r)<\LM_{\leq}(p)$,
thus $r\notin P$ by the minimality of $\LM_{\leq}(p)$,
so that $\overline{r}\in\overline{W}$, a contradiction.

Next we show that $\overline{B}$ is linearly independent over $K$.
Suppose to the contrary that there exist $r\in\NN$ and $\alpha_1,\ldots,\alpha_r\in K\wo\{0\}$
and pairwise distinct $\overline{b}_1,\ldots,\overline{b}_r\in\overline{B}$ such that
$\alpha_1\overline{b}_1+\ldots+\alpha_r\overline{b}_r=\overline{0}$.
Then any respective representatives $b_1,\ldots,b_r\in B$ of
$\overline{b}_1,\ldots,\overline{b}_r$ are pairwise distinct and
$\alpha_1 b_1+\ldots+\alpha_r b_r=\slPhi(y)$ for some $y\in L$.
Of course, $y\ne 0$ as the monomials $b_1,\ldots,b_r$ are linearly independent over $K$.
It follows ${\LM_{\leq}(\slPhi(y))=b_i\in B}$ for some ${i\in\{1,\ldots,r\}}$.
Therefore  ${\LM_{\leq}(\slPhi(y))\in B\cap \LM_{\leq}(\slPhi(L))}$,
that is, ${\LM_\preceq(y)\in B\cap \LM_{\preceq}(L)}$ by \ref{parapapunzi}.
But, by definition, $B\cap \LM_{\preceq}(L)=\emptyset$, a contradiction.
\end{proof}

\begin{cor}\label{lollo}
Let ${\preceq},{\preceq'}\in\WO(N)$,
assume that $A_K^{t,\slPhi}$ is  multiplicative on $\{ {\preceq}, {\preceq'} \}$,
and let $L$ be a left ideal of $A$ with $\LM_{\preceq}(L)\subseteq\LM_{\preceq'}(L)$.
Then $\LM_{\preceq}(L)=\LM_{\preceq'}(L)$.
\end{cor}

\begin{proof}
Put $B=M\wo\LM_{\preceq}(L)$ and $B'=M\wo\LM_{\preceq'}(L)$.
Let 
${\overline{\phantom{\bullet}}:K[X]\rightarrow K[X]\mod\slPhi(L)}$
be 
the residue class homomorphism (of $K$-modules).
Suppose by contradiction that $\LM_{\preceq}(L)\subsetneq\LM_{\preceq'}(L)$.
Then $B\supsetneq B'$, hence $\overline{B}\supseteq\overline{B'}$.

If it held $\overline{B}=\overline{B'}$,
then we would find $b\in B\wo B'$ and $b'\in B'$ such that $\overline{b}=\overline{b'}$,
hence $b-b'\in\slPhi(L)$,
thus $\LM_{\phi({\preceq})}(b-b')\in\LM_{\phi({\preceq})}(\slPhi(L))=\LM_\preceq(L)$;
on the other hand, either $\LM_{\phi({\preceq})}(b-b')=b$ or $\LM_{\phi({\preceq})}(b-b')=b'$,
in any case $\LM_{\phi({\preceq})}(b-b')\in B$, a contradiction.

Thus $\overline{B}\supsetneq\overline{B'}$.
But, by \ref{macaulay}, $\overline{B}$ and $\overline{B'}$ are $K$-bases of $K[X]\mod\slPhi(L)$,
hence the one cannot strictly contain the other, a contradiction.
\end{proof}

\begin{cor}\label{fini3}
Let $\mathfrak{T}\subseteq\WO(N)$ such that $\mathfrak{T}$ is closed in $\TO(N)$,
assume that $A_K^{t,\slPhi}$ is multiplicative on~$\mathfrak{T}$,
and let $L$ be a left ideal of $A$.
Then $\TTOO_\mathfrak{T}(L)=\LLOO_\mathfrak{T}(L)$.
In particular, $\TTOO_\mathfrak{T}(L)$ is finite, that is,
$L$ admits at most finitely many distinct leading monomial ideals from~$\mathfrak{T}$.
\end{cor}

\begin{proof}
By \ref{lollo}, $\mathfrak{T}=\min_L(\mathfrak{T})$,
thus $\TTOO_\mathfrak{T}(L)=\LLOO_\mathfrak{T}(L)$,
which is finite by \ref{fini2}.
\end{proof}

\section{Admissible orderings}\label{sao}

\noindent
We keep the notation of the previous section.


\begin{dfn}\label{cumpa}
A \notion{compatible ordering on $M$} or 
\notion{of $K[X]$} is a total ordering ${\leq}$ on $M$ such that
for all $\upsilon,\nu,\gamma\in\NN_0^t$ it holds compatibility:
${X^\upsilon\leq X^\nu \dann X^{\upsilon+\gamma}\leq X^{\nu+\gamma}}$.


Compatible orderings are also known as \notion{semigroup orderings}.
The set of all compatible orderings of $K[X]$ is denoted by $\CO(M)$.

We also consider the set of \notion{compatible orderings on $N$} or \notion{of $A_K^{t,\slPhi}$}
or simply \notion{of $A$} defined as $\CO(N)=\phi^{-1}(\CO(M))$.
\end{dfn}

\begin{pro}\label{CO(M) compact}
$\CO(M)$ and $\CO(N)$ are closed in $\TO(M)$ and $\TO(N)$, re\-spect\-ive\-ly,
and hence compact.
\end{pro}

\begin{proof}
Let $(S_i)_{i\in\NN_0}$ be a filtration of $M$ consisting of finite sets $S_i$.
Let ${\leq}\in\TO(M)$ be an accumulation point of $\CO(M)$. 
Thus, by definition, for each $r\in\NN_0$ there exists
${{\leq}_r\in\CO(M)\cap\fraktur{N}_r({\leq})\wo\{{\leq}\}}$, 
so that ${\leq}_r$ and ${\leq}$ agree on $S_{r+1}$.
Choose any ${\upsilon,\nu\in\NN_0^t}$ and assume that $X^\upsilon\leq X^\nu$, say.
Let $\gamma\in\NN_0^t$.
Then we find $r\in\NN_0$ such that $S_{r+1}$ contains the monomials
$X^\upsilon,X^\nu,X^{\upsilon+\gamma},X^{\nu+\gamma}$. 
There exists ${\leq}_r$ as above that agrees with ${\leq}$ on $S_{r+1}$,
so that $X^\upsilon\leq_r X^\nu$.
Since ${\leq}_r$ is a compatible ordering of $K[X]$, it follows
$X^{\upsilon+\gamma}\leq_r X^{\nu+\gamma}$.
Therefore 
${X^{\upsilon+\gamma}\leq X^{\nu+\gamma}}$.
Hence ${\leq}\in\CO(M)$.
Thus $\CO(M)$ contains all its accumulation points in $\TO(M)$ and hence $\CO(M)$ is
closed in $\TO(M)$.
Since $\TO(M)$ is compact by \ref{compact}, $\CO(M)$ is compact.
Since $\phi$ is a homeo\-morphism by \ref{homeo}, also $\CO(N)$ is closed in $\TO(N)$
and compact.
\end{proof}

\begin{dfn}
$\AO(M)=\SO_1(M)\cap\CO(M)$
is the set of all \notion{admissible orderings on $M$} or \notion{of $K[X]$},
and $\AO(N)=\SO_{\slPhi^{-1}(1)}(N)\cap\CO(N)$
is the set of all \notion{admissible orderings  on $N$} or 
\notion{of $A_K^{t,\slPhi}$} or simply \notion{of $A$}.
Observe that $\phi^{-1}(\AO(M))=\AO(N)$.
Admissible orderings are also known as \notion{monoid orderings}.
\end{dfn}

\begin{rem}
One sees that this definition of admissible ordering on $M$ and on $N$ is equivalent to the one
given in \cite{Kan}, and it is equivalent to the notion of term orderings given
in \cite{Sai} in the case of Weyl algebras under the assumption that $\slPhi(1)=1$.
\end{rem}

\begin{rem}\label{AO(M)}
An admissible ordering of $K[X]$ is  a total ordering ${\leq}$ on $M$
such that it holds
well-foundedness:~${1\leq X^\nu}$,
and 
compatibility: ${X^\upsilon\leq X^\nu \dann X^{\upsilon+\gamma}\leq X^{\nu+\gamma}}$.
%
Since $M$ is a $K$-basis of $K[X]$, these axioms are equivalent to:
${1 < X^\nu}$
whenever $\nu\ne 0$,
and
${X^\upsilon < X^\nu \dann X^{\upsilon+\gamma} < X^{\nu+\gamma}}$.
\end{rem}

\begin{exa}\label{exa ao}
The \notion{lexicographical ordering} ${\leq_\mathrm{lex}}$ on $M$
defined by
\begin{equation*}
{X^\upsilon\leq_{\mathrm{lex}}X^\nu}
:\gdw
{(\upsilon = \nu)} 
\oder
{(\upsilon \ne \nu \und \upsilon_{m( \upsilon,\nu )} < \nu_{m( \upsilon,\nu )})}
\end{equation*}
for all $\upsilon,\nu\in\NN_0^t$,
where we put $m(\alpha,\beta)=\min{}\{\,k\mid 1\le k\le t \,\wedge\, \alpha_k\ne\beta_k\,\}$
for all $\alpha,\beta\in\NN_0^t$ with $\alpha\ne\beta$,
is an admissible ordering of $K[X]$.
\end{exa}

\begin{exa}\label{exa ao 2}
Fixed any ${\leq}\in\AO(M)$,
for all $\omega\in\NN_0^t$ one can define the \notion{$\omega$-graded ${\leq}$-ordering}
${\leq}_\omega$ by
\begin{equation*}
{X^\upsilon\leq_{\omega}X^\nu}
:\gdw
{(\omega\cdot\upsilon < \omega\cdot\nu)} 
\oder
{(\omega\cdot\upsilon = \omega\cdot\nu \und X^\nu \!\leq Y^\upsilon)}
\end{equation*}
for all $\upsilon,\nu\in\NN_0^t$, and one has that ${\leq}_\omega$ is an admissible ordering 
of $K[X]$, see Exercise~12 in \cite[{II}.4]{Cox}
\end{exa}

\begin{pro}\label{AO(N) compact}
$\AO(M)$ and $\AO(N)$ are closed in $\TO(M)$ and $\TO(N)$,
respectively, and hence compact.
\end{pro}

\begin{proof}
Clear by \ref{CO(M) compact}, \ref{closed}, and \ref{compact}.
\end{proof}

\begin{pro}\label{cox}
$\AO(M)\!=\!\WO(M)\cap\CO(M)\!$ and $\AO(N)\!=\!\WO(N)\cap\CO(N)$.
\end{pro}

\begin{proof}
By \cite[II.4.6]{Cox} one has $\SO_1(M)\cap\CO(M)=\WO(M)\cap\CO(M)$.
Since $\phi^{-1}$ is injective and 
since $\phi^{-1}(\CO(M))=\CO(N)$
and $\phi^{-1}(\SO_1(M))=\SO_{\slPhi^{-1}(1)}(N)$
and $\phi^{-1}(\WO(M))=\WO(N)$,
the second claim follows.
\end{proof}

\section{Degree-compatible orderings}\label{sdo}

\noindent
We keep the notation of the previous section.

\begin{exa}\label{exa do 2}
It holds $\DO(M)\nsubseteq\CO(M)$ and hence $\DO(N)\nsubseteq\CO(N)$.
Indeed, any degree ordering ${\leq}$ of $K[Y,Z]$ such that
$1<Y<Z<YZ<Y^2<Z^2<\ldots$ is not compatible
because compatibility would force $Y^2<YZ$ from $Y<Z$.

Also it holds $\CO(M)\nsubseteq\DO(M)$ and hence $\CO(N)\nsubseteq\DO(N)$.
For instance,
the lexicographic ordering ${\leq_{\mathrm{lex}}}$ of $K[Y,Z]$ induced 
by $Y<_{\mathrm{lex}}Z$ is  compatible  but is not a degree ordering
as ${\deg(\LM_{\leq}(Y+Z^2))=\deg(Y)=1\ne 2=\deg(Y+Z^2)}$.
\end{exa}

\begin{rem+dfn}\label{DCO compact}
It is not to expect that there exist interesting $K$-algebras of countable type
that are multiplicative on $\DO(M)$ since even $K[X]$ is not.
For a degree ordering ${\leq}$ of $K[Y,Z]$ with 
$1<Y<Z<Y^2<Z^2<YZ<\ldots$~for instance, 
it holds $\LT_\leq((Y+Z)^2)=YZ \ne Z^2=\LT_\leq(Y+Z)\LT_\leq(Y+Z)$.

Therefore we shall consider the set $\DCO(M)=\DO(M)\cap\CO(M)$ of the
\notion{degree-compatible orderings on~$M$} or \notion{of $K[X]$} and
the set $\DCO(N)=\DO(N)\cap\CO(N)$ of the
\notion{degree-compatible orderings on $N$} or \notion{of $A_K^{t,\slPhi}$} or simply
\notion{of $A$}.

Of course, it holds $\DCO(N)=\phi^{-1}(\DCO(M))$.
Moreover, $\DCO(M)\subseteq\AO(M)$ by \ref{dco1}, and hence $\DCO(N)\subseteq\AO(N)$.
Finally, by \ref{DO(M) compact} and \ref{WO M N} and by \ref{CO(M) compact},
$\DCO(M)$ and $\DCO(N)$ are closed in $\TO(M)$ and $\TO(N)$, respectively, and compact.
\end{rem+dfn}

\begin{pro}\label{nowhere dense}
If $t>1$, where $t$ is the number of indeterminates,
then $\DCO(M)$ is nowhere dense in $\DO(M)$, and so is $\DCO(N)$ in $\DO(N)$.
\end{pro}

\begin{proof}
Consider the filtration $(S_i)_{i\in\NN_0}$ of $M$ given by $S_i=\{m\in M\mid\deg(m)<i\}$.
Suppose that some ordering ${\leq}$ lies in the interior $\DCO(M)^{\circ}$ of the closed subset
$\DCO(M)$ of $\DO(M)$.
Then we find a neighbourhood of ${\leq}$ open in $\DO(M)$  contained in 
 $\DCO(M)^{\circ}$, that is, we find $r\in\NN_0$ such that
${\mathfrak{N}_r({\leq})\cap\DO(M)\subseteq\DCO(M)}$.
Since $S_1=\{1\}$,
we have $\mathfrak{N}_0({\leq})=\TO(M)$.
As $\DCO(M)\subsetneq\DO(M)$, it follows $r\ge 1$.
Assume that $X_1<X_2$, say.
Then $X_1^{r+2}<X_1^{r+1}X_2$ by compatibility.
Let ${\leq'}$ be the total ordering on $M$ given by
$X_1^{r+1}X_2<' X_1^{r+2}$ and
$m\leq' m' \Leftrightarrow m\leq m'$ whenever
$(m,m')\in M\times M\wo\{(X_1^{r+1}X_2,X_1^{r+2})\}$.
Then ${\leq'}\in\mathfrak{N}_r({\leq})\cap\DO(M)$,
so that ${\leq'}\in\DCO(M)$.
As $r\ge 1$, we have that ${\leq}$ and ${\leq'}$ agree on $S_2$, thus $X_1<'X_2$.
By compatibility it follows $X_1^{r+2}<'X_1^{r+1}X_2$, a contradiction.
Now we conclude by~\ref{homeo}.
\end{proof}

\begin{rem}\label{T1}
If $t=1$,
then $\vert\DO(M)\vert=\vert\DCO(M)\vert=1=\vert\DCO(N)\vert=\vert\DO(N)\vert$,
thus $\DCO(M)=\DO(M)$ and $\DCO(N)=\DO(N)$.
\end{rem}

\begin{exa}\label{DCO exa}
For each ${\leq}\in\AO(M)$ the binary relation ${\leq_{\deg}}$ on $M$ defined by
\begin{equation*}
m \leq_{\deg} m' \Leftrightarrow \deg(m)<\deg(m')\vee (\deg(m)=\deg(m')\wedge m\le m').
\end{equation*}
is a degree-compatible ordering of $K[X]$.
More generally, the admissible orderings of Example \ref{exa ao 2} are degree-compatible
orderings whenever $\omega\ne 0$ or ${\leq}\in\DCO(M)$.
\end{exa}

\begin{rem}\label{dco endlich}
By \ref{fini3},
for each $H\subseteq A$ and each $\fraktur{T}\subseteq\DCO(N)$
the set $\TTOO_\fraktur{T}(H)$ is finite.
In particular, by \ref{exa ao}, \ref{exa ao 2}, and \ref{DCO exa},
the set
$\TTOO_{\DCO(N)}(H)$ is non\-empty and finite.
\end{rem}

\section{$\fraktur{T}$-admissible algebras}

\noindent
We keep the notation of the previous section.


\begin{dfn}\label{dd}
Let $\mathfrak{T}\!\subseteq\!\AO(N)$.
We say that $A_K^{t,\slPhi}$ or simply $A$ is \notion{$\mathfrak{T}$-admissible}
if $A_K^{t,\slPhi}$ is multiplicative on~$\fraktur{T}$.
We say that $A_K^{t,\slPhi}$ or simply $A$ is \notion{admissible}
if $A_K^{t,\slPhi}$ is $\AO(N)$-admissible.
We say that $A_K^{t,\slPhi}$ or simply $A$ is \notion{degree-compatible}
if $A_K^{t,\slPhi}$ is $\DCO(N)$-admissible.
\end{dfn}

\begin{exa}\label{weyl}
In the terminology of \cite{Kan}, every $K$-algebra that is of solvable type with
respect to all admissible orderings is admissible.
This follows indeed from ${\text{\cite[1.5]{Kan}}}$.

For instance, if $K$ has characteristic~$0$, then every Weyl algebra $W$ over $K$
is isomorphic as a $K$-module to
a commutative polynomial ring over $K$, see  \cite[I.2.1]{Cou}, and
$W$ clearly fulfills the axioms \cite[1.2]{Kan}
of an algebra of solvable type for all admissible orderings on its canonical $K$-basis, 
so that $W$ is multiplicative on these orderings by \cite[1.5]{Kan}.
\end{exa}

\begin{exa}\label{U(L)}
If $K$ has characteristic~$0$,
then the universal enveloping algebra $U(\mathfrak{g})$ of any Lie algebra $\mathfrak{g}$
of finite length over $K$ is degree-compatible.
Indeed, let $X=\{x_1,\ldots,x_r\}$ be a finite $K$-basis of $\mathfrak{g}$.
By the Poincar\'e--Birkhoff--Witt Theorem, see 
 2.13, 2.14, 2.22 of \cite[{II}]{Maz},
there exist then a canonical $K$-module monomorphism $h:\mathfrak{g}\hookrightarrow U(\mathfrak{g})$
and a count\-able $K$-basis
$Y={\{y_1^{\nu_1}\cdots y_r^{\nu_r}\mid(\nu_1,\ldots,\nu_r)\in\NN_0^r\}}$ of $U(\mathfrak{g})$
with $y_i=h(x_i)$ such that $[y_j,y_k]=\sum_{1\le i\le r}c_{ijk}y_i$ for some
${c_{ijk}\in K}$.
Thus,  $U(\mathfrak{g})$ is isomorphic as a $K$-module to the commutative polynomial ring
$K[X_1,\ldots X_r]$ by an isomorphism that maps $y_i$ to $X_i$,
and 
the relations $y_ky_j=y_jy_k-\sum_{1\le i\le r}c_{ijk}y_i$ imply by
\cite[1.2 \& 1.5]{Kan} that $U(\mathfrak{g})$ is multiplicative on $\DCO(Y)$.
\end{exa}

\begin{thm}\label{cocco}
Let $\fraktur{T}\!\subseteq\!\AO(N)$ be a closed subset.
Assume that $A_K^{t,\slPhi}$ is $\fraktur{T}$-admissible.
Let $L$ be a left ideal of~$A$.
Then $\TTOO_{\fraktur{T}}(L)$ is finite,
that is, $L$ admits only finitely many distinct leading monomial ideals from~$\fraktur{T}$.
In particular, if $A_K^{t,\slPhi}$ is admissible, then the non\-empty set $\TTOO_{\AO(N)}(L)$
is finite.
\end{thm}

\begin{proof}
It is all clear by \ref{fini3}, \ref{AO(N) compact}, \ref{cox},
and by \ref{homeo}, \ref{exa ao}, \ref{exa ao 2}, \ref{weyl}.
\end{proof}

\begin{rem}
Notice that by \ref{dco endlich}
we already know this result for subspaces $\fraktur{T}$ of $\DCO(N)$
without having to assume that $A$ be multiplicative on $\fraktur{T}$
nor that $L$ be a left ideal.
\end{rem}

\section{Gr\"obner bases}\label{GB}

\noindent
We keep the notation of the previous section.

\begin{dfn}
Let $A_K^{t,\slPhi}$ be an algebra of countable type,
 $L$ be a left ideal of $A$,
 $N$ denote the canonical $K$-basis of $A_K^{t,\slPhi}$,
and  ${\preceq}$ be a total ordering on $N$. 
A \notion{Gr\"obner basis of $L$ with respect to~${\preceq}$}
is a finite subset $G$ of $L$
such that $L=\sum_{g\in G}Ag$ and $\LM_\preceq(L)=\LM_\preceq(G)$.
\end{dfn}

\begin{rem}
The definition of Gr\"obner basis given here is equivalent to the one given in
\cite{Kan} if one restricts to admissible orderings and algebras of solvable type,
see \cite[3.8]{Kan}.

This definition is also equivalent to the one given in \cite{Sai} 
when further restricting to Weyl algebras.

By ${\text{\cite[II.4.2]{Hui}}}$ it is less general than the one given
in \cite[II.3.2(ii)]{Hui},
but it is equivalent to the definition given in \cite[III.1.1]{Hui} when restricting
to admissible orderings
and free $K$-algebras $K\langle X_\lambda\mid\lambda\in\slLambda\rangle$,
$\slLambda$ any index set.
\end{rem}

\begin{dfn}
Let $A_K^{t,\slPhi}$ be an algebra of countable type, 
let $L$ be a left ideal of $A$,
and let $N$ denote the canonical $K$-basis of $A_K^{t,\slPhi}$.

Given any $\fraktur{T}\subseteq\TO(N)$,
we say that a finite subset $U$ of $L$ is a 
\notion{$\fraktur{T}$-universal Gr\"obner basis of $L$}
if $U$ is a Gr\"obner basis of $L$ with respect to all elements of~$\fraktur{T}$.

In the following we call the $\mathfrak{T}$-universal Gr\"obner bases 
in $\mathfrak{T}$-admissible algebras simply \notion{universal Gr\"obner bases}.
\end{dfn}

\noindent
We fix here an algebra $A_K^{t,\slPhi}$ of countable type
and as usually denote its canonical $K$-basis by $N$.

\begin{thm}\label{GB exist 1}
Assume that $A$ is left noetherian,
let $L$ be a left ideal of $A$,
and let ${\preceq}$ be a total ordering on $N$.
Then $L$ admits a Gr\"obner basis with respect to~${\preceq}$.
\end{thm}

\begin{proof}
Suppose that $L$ admits no Gr\"obner basis with respect to ${\preceq}$.
Since $A$ is left noetherian, there exists a finite subset $F_0$ of $L$ such that $L=AF_0$.
It holds ${\LT_\preceq(F_0)\subsetneq\LT_\preceq(L)}$
as $F_0$ is not a Gr\"obner basis. 
Thus there exists $x_1\in L\wo\{0\}$ with ${\LT_\preceq(x_1)\notin \LT_\preceq(F_0)}$.
Put ${F_1=F_0\cup \{x_1\}}$.
Again 
${\LT_\preceq(F_1)\subsetneq\LT_\preceq(L)}$
as $F_1$ is not a Gr\"obner basis. 
Thus there exists $x_2\in L\wo \{0\}$ with ${\LT_\preceq(x_2)\notin \LT_\preceq(F_1)}$.
Put ${F_2=F_1\cup \{x_2\}}$.
Again 
${\LT_\preceq(F_2)\subsetneq\LT_\preceq(L)}$
as $F_2$ is not a Gr\"obner basis$\ldots$ 
We construct in this way an infinite chain
$\LT_\preceq(F_0)\subsetneq \LT_\preceq(F_1)\subsetneq \LT_\preceq(F_2)\subsetneq\ldots$~of
ideals of $K[X]$, in contradiction to the noetherianity of $K[X]$.
\end{proof}

\begin{thm}\label{gen}
Assume that there exists 
${\preceq}\in\WO(N)$ with the property that $A_K^{t,\slPhi}$ is multiplicative
on~$\{\preceq\}$.
Let $L$ be a left ideal of $A$ and $F$ be a finite subset of $L$ such that
$\LT_{\preceq}(L)=\LT_{\preceq}(F)$.
Then $L=\sum_{f\in F}Af$. 
\end{thm}

\begin{proof}
Trivially, we have $\sum_{f\in F}Af\subseteq L$.
Suppose that $\sum_{f\in F}Af\subsetneq L$.
Then the set $U=\{\LT_{\preceq}(l)\mid l\in L\wo\sum_{f\in F}Af\}$ is non\-empty.
We have ${\leq}=\phi({\preceq})\in\WO(M)$, and so there exists
$l\in L\wo\sum_{f\in F}Af$ such that $u=\LT_{\preceq}(l)$ is minimal in $U$ with
respect to~${\leq}$.
Since $u\in\LT_{\preceq}(L)=\LT_{\preceq}(F)$, we can write
$u=\sum_{f\in F\wo\{0\}}p_f\LT_{\preceq}(f)$ for some family
$(p_f)_{f\in F\wo\{0\}}$ of polynomials. 
As $u\in M$ and $M$ is a $K$-basis of $K[X]$,
we find
$m\in \bigcup_{f\in F\wo\{0\}}\Supp(p_f)\subseteq M$ and $g\in F\wo\{0\}$
such that ${u=m\LT_{\preceq}(g)}$. 
Put $n=\slPhi^{-1}(m)$.
As $n\in N$, clearly $n\ne 0$.
Since $A$ is a domain, it follows $ng\ne 0$.
Now put $\smash{h=l-\lc_{\preceq}(l)\lc_{\preceq}(ng)^{-1}ng}$.
Then $h\in L\wo\sum_{f\in F}Af$, thus $\LT_{\preceq}(h)\in U$.
On the other hand,
$\LT_{\preceq}(ng)=\LT_{\preceq}(n)\LT_{\preceq}(g)=m\LT_{\preceq}(g)=u=\LT_{\preceq}(l)$,
so that $\LT_{\preceq}(h)<\LT_{\preceq}(l)$, a contradiction.
\end{proof}

\begin{cor}\label{cucu}
Assume that there exists ${\preceq}\in\WO(N)$ such that $A_K^{t,\slPhi}$ is multiplicative
on~$\{\preceq\}$.
Then $A$ is left noetherian.
\end{cor}

\begin{proof}
Let $L$ be a left ideal of $A$.
As $K[X]$ is noetherian, 
we find a finite subset $F$ of $L$ such that $\LM_\preceq(F)=\LM_\preceq(L)$.
By \ref{gen}, $F$ is a generating set of $L$.
Thus every left ideal of $A$ is finitely generated. 
\end{proof}

\begin{cor}\label{esistenza}
Assume that there exists ${\preceq}\in\WO(N)$ such that $A_K^{t,\slPhi}$ is multiplicative
on $\{\preceq\}$.
Then for each left ideal $L$ of $A$ and each total ordering ${\preceq'}$ on $N$ there exists
a Gr\"obner basis of $L$ with respect to ${\preceq'}$.
\end{cor}

\begin{proof}
Clear by \ref{GB exist 1} and \ref{cucu}.
\end{proof}



% ...

\section{Universal Gr\"obner bases in admissible algebras}

\noindent
We keep the notation of the previous section.

\begin{lem}\label{qwert1}
Let ${\preceq},{\preceq'}\in\WO(N)$ such that
$A_K^{t,\slPhi}$ is multiplicative on $\{{\preceq},{\preceq'}\}$.
Let $L$ be a left ideal of $A$ and  $G$ be a Gr\"obner basis of $L$ with respect
to~${\preceq}$.
If ${\preceq}$ and ${\preceq'}$ agree on $\Supp(G)$,
then $\LM_{\preceq}(L)=\LM_{\preceq'}(L)$
and $G$ is a Gr\"obner basis of $L$ with respect to~${\preceq'}$.
\end{lem}

\begin{proof}
Because ${\preceq}$ and ${\preceq'}$ agree on $\Supp(G)$,
it follows that $\phi({\preceq})$ and $\phi(\preceq')$ agree
on $\slPhi(\Supp(G))=\Supp(\slPhi(G))$.
Hence
$\LM_{\phi({\preceq})}(\slPhi(G))=\LM_{\phi(\preceq')}(\slPhi(G))$
by \ref{LT=LT2}.
From \ref{parapapunzi} it follows
$\LM_\preceq(L)=\LM_{\preceq}(G)=\LM_{\preceq'}(G)\subseteq\LM_{\preceq'}(L)$.
As $\TO(N)$ is a Hausdorff space, see \ref{U=N},
points are closed, 
so $\{{\preceq},{\preceq'}\}$ is closed in $\TO(N)$.
Thus $\TTOO_{\{{\preceq},{\preceq'}\}}(L)=\LLOO_{\{{\preceq},{\preceq'}\}}(L)$ by \ref{fini3},
and hence $\LM_\preceq(L)=\LM_{\preceq'}(L)$,
and therefore $\LM_{\preceq'}(G)=\LM_{\preceq'}(L)$.
\end{proof}

\begin{lem}\label{aperto1}
Let $\fraktur{T}\subseteq\WO(N)$ such that $A_K^{t,\slPhi}$ is multiplicative
on~$\mathfrak{T}$.
Let $L$ be a left ideal of $A$ and let $F$ be a finite subset of $L$.
Then the set $\mathfrak{U}_{L}(F)$ of all ${\preceq}\in\fraktur{T}$ such
that $F$ is a Gr\"obner basis of $L$ with respect to~${\preceq}$ is open in~$\fraktur{T}$.
\end{lem}

\begin{proof}
Let $(S_i)_{i\in\NN_0}$ be a filtration of $N$ consisting of finite sets $S_i$.
There exists $r\in\NN_0$ such that the finite subset $\Supp(F)$ of $N$ lies in~$S_{r+1}$.
We may assume that $\mathfrak{U}_{L}(F)\ne\emptyset$, so that $\mathfrak{T}\ne\emptyset$.
Let ${\preceq}\in\mathfrak{U}_{L}(F)$.
Thus $F$ is a Gr\"obner basis of $L$ with respect to~${\preceq}$.
Consider the open neighbourhood $\mathfrak{N}_r({\preceq})\cap\fraktur{T}$ of
${\preceq}$ in $\fraktur{T}$ and let ${\preceq'}\in\mathfrak{N}_r({\preceq})\cap\fraktur{T}$.
Then ${\preceq}$ and ${\preceq'}$ agree on $S_{r+1}$ and in particular on $\Supp(F)$.
By \ref{qwert1}, $F$ is a Gr\"obner basis of $L$ with respect to~${\preceq'}$,
that is, ${\preceq'}\in\mathfrak{U}_{L}(F)$.
Hence ${\preceq}\in\mathfrak{N}_r({\preceq})\cap\fraktur{T}\subseteq\mathfrak{U}_{L}(F)$,
and $\mathfrak{U}_{L}(F)$ is open in~$\fraktur{T}$.
\end{proof}

\begin{rem}\label{ricoprimento1}
Let $\emptyset\ne\mathfrak{T}\subseteq\WO(N)$ 
such that $A_K^{t,\slPhi}$ is multiplicative on~$\mathfrak{T}$.
Let $L$ be a left ideal of $A$.
Then, by \ref{esistenza}, for each ${\preceq}\in\fraktur{T}$
there exists a Gr\"obner basis $G_{\preceq}$ of $L$ with respect to~${\preceq}$.
Thus, in the notation of \ref{aperto1},
clearly ${\preceq}\in\mathfrak{U}_{L}(G_{\preceq})$
for each ${\preceq}\in\mathfrak{T}$.
Hence, by \ref{aperto1}, $\bigcup_{{\preceq}\in\fraktur{T}}\mathfrak{U}_{L}(G_\preceq)$
is an open covering of~$\fraktur{T}$.
\end{rem}

\begin{thm}\label{UGB exist 1}
Let $\emptyset\ne\fraktur{T}\subseteq\WO(N)$ such that $\mathfrak{T}$ is closed in $\TO(N)$
and $A_K^{t,\slPhi}$ is multiplicative on~$\fraktur{T}$.
Let $L$ be a left ideal of $A$.
Then $L$ admits a $\fraktur{T}$-universal Gr\"obner basis.
\end{thm}

\begin{proof}
In the notation of \ref{ricoprimento1},
$\bigcup_{{\preceq}\in\fraktur{T}}\fraktur{U}_{L}(G_\preceq)$ is an open
covering of~$\mathfrak{T}$, where each $G_\preceq$ is a Gr\"obner basis of $L$ with respect
to~${\preceq}$.
As $\TO(N)$ is compact and $\fraktur{T}$ is closed in $\TO(N)$, 
$\fraktur{T}$ is compact.
Hence we can find $s\in\NN$ and ${\preceq_1},\ldots,{\preceq_s}\in\fraktur{T}$ such that
$\bigcup_{1\le j\le s}\fraktur{U}_{L}(G_{\preceq_j})$
is a finite open covering of~$\fraktur{T}$.
We claim that $U=\bigcup_{1\le j\le s} G_{\preceq_j}$ is a $\fraktur{T}$-universal Gr\"obner
basis of $L$.
Indeed, let ${\preceq_0}\in\fraktur{T}$.
Then there exists $j\in\{1,\ldots,s\}$ such that
${\preceq_0}\in\fraktur{U}_{L}(G_{\preceq_j})$.
Thus $G_{\preceq_j}$ is a Gr\"obner basis of $L$ with respect to~${\preceq_0}$.
As $G_{\preceq_j}\subseteq U$, 
of course also $U$ is a Gr\"obner basis of $L$ with respect to~${\preceq_0}$.
Since the choice of ${\preceq_0}$ in $\fraktur{T}$ was arbitrary, we conclude that
$U$ is a $\fraktur{T}$-universal Gr\"obner basis of $L$.
\end{proof}

\begin{cor}\label{ao dco ugb}
Let $\mathfrak{T}$ be a non\-empty closed subset of $\AO(N)$ 
such that $A_K^{t,\slPhi}$ is $\mathfrak{T}$-admissible.
Then for each left ideal $L$ of $A$ there exists a $\mathfrak{T}$-universal Gr\"obner basis
of $L$.
In particular, every left ideal of an admissible or degree-compatible algebra has
a universal Gr\"obner basis.
\hfill\qedsymbol
\end{cor}

\begin{rem}\label{denso1}
To effectively compute a $\fraktur{T}$-universal Gr\"obner basis, one should start walking
among the orderings in $\fraktur{T}$ and pick some ones that allow to cover $\fraktur{T}$ as
in \ref{ricoprimento1}.
But how to pluck the right flowers in that vast meadow?
The following Lemma \ref{denso2} \emph{might} be of help.
Once one thinks to have located a suitable kind of orderings,
that is, an appropriate subset $\fraktur{D}$ of $\fraktur{T}$,
if one is able to show that $\fraktur{D}$ is dense in $\fraktur{T}$,
then one can indeed restrict the own search to $\fraktur{D}$.
This fact might be the first step toward the construction of a ``topological algorithm''
that computes a $\fraktur{T}$-universal Gr\"obner basis.
\end{rem}

\begin{lem}\label{denso2}
In the hypotheses of \ref{UGB exist 1}, let $\fraktur{D}$ be a dense subset of~$\fraktur{T}$.
Then we can find finitely many ${\preceq}_1,\ldots,{\preceq}_s$ in $\fraktur{D}$
and respective Gr\"obner bases $G_1,\ldots,G_s$ of $L$
such that $\bigcup_{1\le j\le s}G_{j}$ is a $\fraktur{T}$-universal Gr\"obner basis of~$L$.
\end{lem}

\begin{proof}
Because $\fraktur{T}$ is compact, we can find finitely many 
${\preceq'_1},\ldots,{\preceq'_s}\in\fraktur{T}$ such that
${\fraktur{T}=\bigcup_{1\le j\le s}\fraktur{U}_L(G_j)}$,
where each $G_j$ is a Gr\"obner basis of $L$ with respect to ${\preceq'_j}$.
Then $\bigcup_{1\le j\le s}G_{j}$ is a $\fraktur{T}$-universal Gr\"obner basis of $L$.

Because $\fraktur{D}$ is dense in $\fraktur{T}$
and each $\fraktur{U}_L(G_j)$ is an open  neighbourhood of ${\preceq'_j}$ in $\fraktur{T}$,
for ${1\le j\le s}$ we find ${{\preceq_j}\in \fraktur{D}\cap\fraktur{U}_L(G_j)}$.
Thus each $G_j$ is a Gr\"obner basis of $L$ with respect to~${\preceq_j}$.
\end{proof}

\begin{exa}\label{denso3}
The orderings ${\preceq}$ given by
\begin{equation*}
\slPhi^{-1}(X^\upsilon)\preceq\slPhi^{-1}(X^\nu) \Leftrightarrow
 X^{\slGamma\upsilon} \leq_{\mathrm{lex}} X^{\slGamma\nu}
\end{equation*}
with $\slGamma$  a ${t\times t\text{-}}$matrix with entries in $\NN_0$
constitute a dense subset of $\AO(N)$.
This follows easily from \cite[p.~6]{Ley}.
\end{exa}











\begin{dfn}
Let $(X,d)$ be a metric space and let $\varepsilon\in\RR$ with $\varepsilon>0$.
We say that $Y\subseteq X$ is \notion{$\varepsilon$-dense in $X$}
if for all $x\in X$ there exists $y\in Y$ such that
$d(x,y)<\varepsilon$.
\end{dfn}

\begin{lem}
In the hypotheses of \ref{UGB exist 1},
assume that there exists $r\in\NN_0$ such that for all ${\preceq}\in\fraktur{T}$ and
all Gr\"obner bases $G_\preceq$ of $L$ with respect to ${\preceq}$
and all $g\in G_\preceq$ it holds $\deg(\slPhi(g))\le r$.
Let $\mathbb{S}=(S_i)_{i\in\NN_0}$
be the filtration of $N$ with ${S_i=\slPhi^{-1}(M_{\le i-1})}$.
Let $\fraktur{D}$ be a $\frac{1}{r}$-dense subset of $\,\fraktur{T}$ with respect to
the metric $d_\mathbb{S}\res_\mathfrak{T}$ induced by $\mathbb{S}$.
Then we can find finitely many ${\preceq}_1,\ldots,{\preceq}_s$ in $\fraktur{D}$
and respective Gr\"obner bases $G_1,\ldots,G_s$ of $L$
such that $\bigcup_{1\le j\le s}G_{j}$ is a $\fraktur{T}$-universal Gr\"obner basis of~$L$.
\end{lem}

\begin{proof}
We find $s\in\NN$ and ${\preceq'_1},\ldots,{\preceq'_s}\in\fraktur{T}$ and
$G_1,\ldots,G_s\subseteq L$
such that each $G_j$ is a ${\preceq'_j}$-Gr\"obner basis of $L$ and
$U=\bigcup_{1\le j\le s}G_j$ is a $\fraktur{T}$-universal Gr\"obner basis of~$L$.

It holds $\Supp(U)\subseteq S_{r+1}$.
Because $\fraktur{D}$ is $\frac{1}{r}$-dense in $\fraktur{T}$,
for $1\le j\le s$ there exists ${\preceq_j}\in\fraktur{D}\cap\fraktur{N}_r(\preceq'_j)$.
Since ${\preceq'}_j$ and ${\preceq_j}$ agree on $\Supp(U)$ and hence on $\Supp(G_j)$,
by \ref{qwert1} 
$G_j$ is a Gr\"obner basis of $L$ with respect to~${\preceq_j}$.
\end{proof}

\begin{rem}\label{degree bound}
Assume that $A_K^{t,\slPhi}$ is a quadric algebra of solvable type,
this means, 
$\slPhi^{-1}(X_i)\slPhi^{-1}(X_j)=\slPhi^{-1}(X_j)\slPhi^{-1}(X_i)+\slPhi^{-1}(p_{ij})$
for polynomials ${p_{ij}\in K[X]}$ at most of degree $2$.
Assume further that $L$ can be generated by finitely many
elements $x_1,\ldots,x_q$ such that $\deg(\slPhi(x_h))\le d$ for $1\le h\le q$.
As proved in \cite{Ley}, for each ${{\preceq}\in\AO(N)}$
there exists a Gr\"obner basis $G_\preceq$ of $L$ with respect to ${\preceq}$
such that $\deg(\slPhi(g))\le 2(\frac{d^2+2d}{2})_{}^{2^{t-1}}$ for all $g\in G_\preceq$.
Therefore
there exists a $\fraktur{T}$-universal Gr\"obner basis $U$ of $L$ such that
${\deg(\slPhi(u))\le 2(\frac{d^2+2d}{2})_{}^{2^{t-1}}}$ for all $u\in U$,
for one can construct $U$ as a union of (finitely~many) such Gr\"obner bases $G_\preceq$.
\end{rem}

\begin{rem}
An alternative, ``classical'' proof of \ref{ao dco ugb} involves a division and a reduction
algorithm: 
\begin{enumerate}[\upshape (i)]\formatenum
\item\label{1}
Assume that $A_K^{t,\slPhi}$ is multiplicative on $\{\preceq\}$ for some ${\preceq}\in\WO(N)$.
Let $a\in A$,  $F\subseteq L$ be finite, and  ${\leq}=\phi({\preceq})$.
Then there exist $r\in A$ and $(q_f)_{f\in F}\in A^{\oplus F}$ such that:
\begin{enumerate}[\upshape (a)]\formatenum
\item
$a=\sum_{f\in F}q_f+r$, 
\item
$\forall\,f\in F : (f\ne 0\Rightarrow \forall\,s\in\Supp(r) : \LT_\preceq(f)\nmid\slPhi(s))$, 
\item
$a\ne 0\Rightarrow
(\forall\,f\in F : (q_f f\ne 0\Rightarrow \LT_\preceq(q_f f)\leq\LT_\preceq(a)))$.
\end{enumerate}
\item\label{2}
Let ${\preceq}\in\AO(N)$ such that $A_K^{t,\slPhi}$ is multiplicative on $\{\preceq\}$.
Let $L$ be a left ideal of $A$.
Let $G$ be a Gr\"obner basis of $L$ with respect to~${\preceq}$.
One can then transform $G$ by applying repeatedly the following procedures:
\begin{enumerate}[\upshape (a)]\formatenum
\item
If there exists $g\in G\wo\{0\}$ such that $\LT_\preceq(g)\in\LT_\preceq(G\wo\{g\})$,
then replace $G$ by $G\wo\{g\}$.
\item
If there exist $g\in G\wo\{0\}$ and $n\in\Supp(g)\wo\{\LT_\preceq(g)\}$ such that
$n\in\LT_\preceq(G\wo\{g\})$,
then divide $g$ by $G\wo\{g\}$ as in \eqref{1},
so that it holds ${g=\sum_{f\in G\wo\{g\}}q_f f+r}$,
and replace $G$ by $(\{r\}\cup G)\wo\{g\}$, 
which is equal to $\{r\}\cup (G\wo\{g\})$ in this case.
\end{enumerate}
After finitely many steps both conditions become false, and the process halts with a
\notion{reduced} Gr\"obner basis $G$ of $L$ with respect to ${\preceq}$, that is,
for each $g\in G$ and each $n\in\Supp(g)$ it holds $n\notin\LM_\preceq(G\wo\{g\})$.
\item
Let $\fraktur{T}$ be a closed subset of $\AO(N)$
such that $A_K^{t,\slPhi}$ is $\fraktur{T}$-admissible.
Let $L$ be a left ideal of $A$.
Then there exist at most finitely many leading monomial ideals of $L$
from $\fraktur{T}$, 
thus we find a finite subset $\fraktur{U}$ of $\fraktur{T}$
such that $\TTOO_{\fraktur{U}}(L)=\TTOO_\fraktur{T}(L)$.
For each ${\preceq}\in\fraktur{U}$ we may choose a reduced Gr\"obner basis $G_\preceq$ of $L$
with respect to~${\preceq}$.
Then $\bigcup_{{\preceq}\mathop{\in}\fraktur{U}}G_\preceq$ is a 
$\fraktur{T}$-universal Gr\"obner basis of~$L$.
\end{enumerate}
\end{rem}

\section{Universal Gr\"obner bases from degree orderings}

\noindent
We keep the notation of the previous section.

\begin{lem}\label{aperto2}
Let $L$ be a left ideal of $A$,
let $F$ be a finite subset of $L$,
and let $\fraktur{T}$ be a subspace of ${}\DO(N)$.
Then the set $\mathfrak{U}_{L}(F)$ of all ${\preceq}\in\fraktur{T}$ 
such that $F$ is a Gr\"obner basis of $L$ with respect to~${\preceq}$ is open in~$\fraktur{T}$.
\end{lem}

\begin{proof}
We may assume that $\mathfrak{U}_{L}(F)\ne\emptyset$. 
Let ${\preceq}\in\mathfrak{U}_{L}(F)$.
Thus $F$ is a Gr\"obner basis of $L$ with respect to~${\preceq}$, that is, it holds
${L=\sum_{f\in F}Af}$ and ${\LM_{\preceq}(F)=\LM_\preceq(L)}$.
Put ${\leq}=\phi({\preceq})$ and $E=\slPhi(F)$ and $J=\slPhi(L)$.
Of course, ${\leq}\in\DO(M)$.
Hence, by \ref{llll}, we can find
open neighbourhoods $\mathfrak{V}_E$ and $\mathfrak{V}_J$ of ${\leq}$ in $\DO(M)$
such that
${\LM_{\leq'}(E)=\LM_{\leq}(E)}$ for all ${{\leq'}\in\mathfrak{V}_E}$
and
$\LM_{\leq'}(J)=\LM_{\leq}(J)$ for all ${{\leq'}\in\mathfrak{V}_J}$.
By \ref{parapapunzi}
it follows
$\LM_{\leq'}(E)=\LM_{\leq}(E)=\LM_{\preceq}(F)=\LM_{\preceq}(L)=\LM_{\leq}(J)=\LM_{\leq'}(J)$
for all ${\leq'}\in\mathfrak{V}$,
where 
$\mathfrak{V}=\mathfrak{V}_E\cap\mathfrak{V}_J$.
Put ${\mathfrak{W}=\phi^{-1}(\mathfrak{V})\cap\mathfrak{T}}$.
By \ref{homeo}, $\mathfrak{W}$ is an open subset of $\mathfrak{T}$ such that
${\preceq}\in\mathfrak{W}$.
Again by \ref{parapapunzi} we obtain
$\LM_{\preceq'}(F)=\LM_{\phi({\preceq'})}(E)=\LM_{\phi({\preceq'})}(J)=\LM_{\preceq'}(L)$
for all ${\preceq'}\in\mathfrak{W}$.
Thus $\mathfrak{W}\subseteq\mathfrak{U}_L(F)$.
Hence $\mathfrak{W}$ is an open neighbourhood of ${\preceq}$ in $\mathfrak{U}_L(F)$.
\end{proof}

\begin{rem}\label{ricoprimento2}
Assume that $A$ is left noetherian,
let $L$ be a left ideal of $A$,
and let $\fraktur{T}$ be a subset of $\DO(N)$.
Then, by \ref{GB exist 1}, for each ${\preceq}\in\fraktur{T}$
there exists a Gr\"obner basis $G_{\preceq}$ of $L$ with respect to~${\preceq}$.
Of course, in the notation of \ref{aperto2}, for each ${\preceq}\in\fraktur{T}$
it holds ${\preceq}\in\mathfrak{U}_{L}(G_\preceq)$,
and thus $\bigcup_{{\preceq}\in\fraktur{T}}\mathfrak{U}_{L}(G_\preceq)$
is an open covering of~$\fraktur{T}$.
\end{rem}

\begin{thm}
Assume that $A$ is left noetherian,
let $L$ be a left ideal of $A$,
and let $\mathfrak{T}$ be a closed subset of ${}\DO(N)$.
Then $L$ admits a $\mathfrak{T}$-universal Gr\"obner basis.
\end{thm}

\begin{proof}
In the notation of \ref{ricoprimento2},
$\bigcup_{{\preceq}\in\fraktur{T}}\fraktur{U}_{L}(G_\preceq)$ is an open
covering of~$\mathfrak{T}$, where each $G_\preceq$ is a Gr\"obner basis of $L$ with respect
to~${\preceq}$.
As $\DO(N)$ is compact and $\fraktur{T}$ is closed in $\DO(N)$, $\fraktur{T}$ is compact.
Hence we can find $s\in\NN$ and ${\preceq_1},\ldots,{\preceq_s}\in\fraktur{T}$ such that
$\bigcup_{1\le j\le s}\fraktur{U}_{L}(G_{\preceq_j})$
is a finite open covering of~$\fraktur{T}$.
We claim that $U=\bigcup_{1\le j\le s} G_{\preceq_j}$ is a $\fraktur{T}$-universal Gr\"obner
basis of $L$.
Indeed, let ${\preceq_0}\in\fraktur{T}$.
Then there exists $j\in\{1,\ldots,s\}$ such that
${\preceq_0}\in\fraktur{U}_{L}(G_{\preceq_j})$.
Thus $G_{\preceq_j}$ is a Gr\"obner basis of $L$ with respect to~${\preceq_0}$.
Hence, clearly, also $U$ is a Gr\"obner basis of $L$ with respect to~${\preceq_0}$.
As the choice of ${\preceq_0}$ in $\fraktur{T}$ was arbitrary, we conclude that
$U$ is a $\fraktur{T}$-universal Gr\"obner basis of $L$.
\end{proof}

\noindent
We have obtained another proof of the result of \ref{ao dco ugb} about
degree-compatible algebras, this time without appealing to the Macaulay Basis Theorem.

\begin{cor}
Left ideals of a degree-compatible algebra always admit a universal Gr\"obner basis.
\hfill\qedsymbol 
\end{cor}

%% file: fig_subspaces.tex
{
\begin{pspicture}(0,0)(12,9)
\psset{fillstyle=none}
\psframe(1,1)(11,8)
\uput[-45](1,1){$\TO(M)$}
\psframe(2,3)(10,6)
\uput[-45](2,3){$\WO(M)$}
\psframe(4,3)(10,7)
\uput[-45](4,3){$\SO_1(M)$}
\psframe(7,2)(10,6)
\uput[-45](7,2){$\CO(M)$}
\psframe[fillstyle=vlines](5,4)(9,5)
\psframe[fillstyle=hlines](7,3)(10,6)
\uput[-45](5,4){$\DO(M)$}
\uput[-45](7,3){$\AO(M)$}
\psframe[fillstyle=solid,fillcolor=white,linecolor=white](7.1,3.5)(8.6,3.9)
\uput[-45](7,4){$\DCO(M)$}
\end{pspicture}
}